 \documentclass[12pt]{amsart}
\usepackage{amsmath}
\usepackage{amsxtra}
\usepackage{amstext}
\usepackage{amssymb}
\usepackage{amsthm}
\usepackage{xcolor}
\usepackage{dsfont}
\usepackage[toc,page]{appendix}

\usepackage[backref=page]{hyperref}
\hypersetup{colorlinks=true}

\usepackage{mathtools}

\newtheorem{thm}{Theorem}[section]
\newtheorem{thmx}{Theorem}

\newtheorem{prop}[thm]{Proposition}
\newtheorem{lemma}[thm]{Lemma}
\newtheorem{mainlemma}[thm]{Main Lemma}

\newtheorem{cor}[thm]{Corollary}

\theoremstyle{definition}
\newtheorem{definition}[thm]{Definition}
\newtheorem*{definition*}{Definition}

\newtheorem{remark}[thm]{Remark}
\numberwithin{equation}{section}
\DeclareMathOperator{\dist}{\operatorname{dist}} % The distance.
\DeclareMathOperator{\supp}{\operatorname{supp}}

\DeclareMathOperator{\lip}{\operatorname{Lip}}
\def\R{\mathbb{R}}
\def\eps{\varepsilon}

\def\wh{\widehat}
\def\wt{\widetilde}

\def\XXint#1#2#3{{\setbox0=\hbox{$#1{#2#3}{%
\int}$ }
\vcenter{\hbox{$#2#3$ }}\kern-.6\wd0}}

\def\I{\mathcal{I}}
\def\H{\mathcal{H}}
\def\R1{\widetilde{R}}
\def\T1{\widetilde{T}}
\def\dist{\operatorname{dist}}
\def\supp{\operatorname{supp}}
\def\Lip{\operatorname{Lip}}
\def\eps{\varepsilon}
\def\kap{\varkappa}
\def\R{\mathbb{R}}

\def\diam{\operatorname{diam}}

\def\osc{\operatorname{osc}}

\def\wh{\widehat}
\def\wt{\widetilde}

\def\S{\mathcal{S}}
\def\I{\mathcal{I}}
\def\kap{\varkappa}

\def\F{\mathcal{F}}
\def\C{\mathbb{C}}
\def\I{\mathcal{I}}
\def\Z{\mathcal{Z}}

\def\F{F}

\def\pip{\pi^\perp}

\def\HT{\widehat{T}}
\def\D{\mathcal{D}}
\def\A{\mathcal{A}}
\def\kalpha{\alpha^{(k)}}
\def\J{30}
\def\bdary{\lambda}

\def\mtilde{\widetilde{\mu}}

\def\XXint#1#2#3{{\setbox0=\hbox{$#1{#2#3}{\int}$}
     \vcenter{\hbox{$#2#3$}}\kern-.5\wd0}}

\def\G{\mathcal{G}}

\title[Principal Value Integrals ]{The Huovinen transform and rectifiability of measures}

\author{Benjamin Jaye}
\email{bjaye3@gatech.edu}
\address{School of Mathematical Sciences, Georgia Tech}

\author{Tom\'as Merch\'an}
\email{tmerchan@umn.edu}
\address{School of Mathematics, University of Minnesota}

\thanks{Research supported in part by NSF DMS-2049477 and DMS-2103534, and by the Clemson Research Foundation.  This work was completed while the first author was in residence at the Hausdorff Institute in Bonn as part of the trimester The Interplay between High-Dimensional Geometry and Probability. The second author was partly supported by the NSF RAISE-TAQS grant DMS-1839077 and the Simons foundation grant 563916, SM}

\begin{document}

\maketitle

\begin{abstract}  For a set $E$ of positive and finite length, we prove that if the Huovinen transform (the convolution operator with kernel $z^k/|z|^{k+1}$ for an odd number $k$) associated to $E$ exists in principal value, then $E$ is rectifiable.

\end{abstract}

\section{Introduction}

We say that a subset $E \subset \mathbb{C}$ is rectifiable if there exists  Lipschitz maps $f_i : \R \to \mathbb{C}$, $i=1,2,...$, such that 
$$\mathcal{H}^1\Bigl(E \setminus \bigcup_{i=1}^\infty f_i(\R)\Bigl) = 0,$$
where $\mathcal{H}^1$ stands for the $1$-dimensional Hausdorff measure.  A locally finite Borel measure $\mu$ on $\mathbb{C}$ is rectifiable if there exists a rectifiable set $E \subset \mathbb{C}$ such that 
$$\mu(\mathbb{C} \setminus E)=0.$$
The goal of this paper is to prove the following result.

\begin{thm}\label{PVthm}  Fix an odd number $k\in \mathbb{N}$.  Suppose that $\mu$ is a finite Borel measure for which
\begin{equation}\label{limsupmeas}\limsup_{r\to 0}\frac{\mu(B(z,r))}{r}\in (0,\infty) \text{ for }\mu\text{-a.e. }z\in \C. 
\end{equation}
If the limit
\begin{equation}\label{huovlimit}\lim_{r \to 0}\int_{|z-\omega|>r}\frac{(z-\omega)^k}{|z-\omega|^{k+1}}d\mu(\omega) \text{ exists for }\mu\text{-a.e. }z\in \C,
\end{equation}
then $\mu$ is rectifiable.
\end{thm}

If $k=1$ then Theorem \ref{PVthm} was proved by Tolsa \cite{To3} using the Menger-Melnikov curvature method\footnote{building upon a number of important results including  \cite{Me,MMV, Dav3, DM, MM, NTV2}}; in this case the principal value integral is the Cauchy transform of the measure $\mu$.  The curvature method is no longer directly applicable to this problem for $k\geq 3$, and it had been an open problem as to whether Theorem \ref{PVthm} holds in this case (see for instance \cite{To5}).

If one replaces the limsup condition in (\ref{limsupmeas}) with the condition of positive lower density
\begin{equation}\label{pld}
\liminf_{r\to 0}\frac{\mu(B(z,r))}{r}>0 \text{ for }\mu\text{-a.e. }z\in \C,
\end{equation} then the case $k=1$ of Theorem \ref{PVthm} was proved earlier by Mattila \cite{M}, and subsequently for all $k$ odd by Huovinen \cite{H}.  It is for this reason that we call the integral transform given by convolution of a measure with the singular kernel $z\mapsto \frac{z^k}{|z|^{k+1}}$ \emph{the Huovinen transform}.\\

Under the assumption (\ref{pld}),  much stronger criteria for rectifiability are available in terms of tangent measures that no longer hold under the condition  (\ref{limsupmeas}), see  \cite[Section 5.8]{P}.  Nevertheless, the tools introduced by Mattila and Huovinen are essential to our method\footnote{More precisely, these techniques play an important role in Theorem \ref{huovSLA} below.}.\\

A natural higher dimensional generalization of Mattila's result for $k=1$ was proved by Mattila-Preiss \cite{MP}, who showed that if $d \in \mathbb{Z} \cap [2,\infty)$, $s\in \mathbb{Z}\cap [1,d-1]$ and $0 < \liminf_{r\to 0}\frac{\mu(B(x,r))}{r^s} < \infty$ for $\mu$-a.e. $x\in \R^d $, then the existence of the $s$-Riesz transform in principal value implies $s$-rectifiability\footnote{ We say that a Borel measure $\mu$ is $s$-rectifiable if there exist Lipschitz maps $f_i: \R^s \to \R^d$, $i=0,1,...$, such that 
$\mu(\R^n \setminus \bigcup_{i=0}^\infty f_i(\R^s))=0$.}.  Here the Riesz transform is the convolution of a measure in $\R^d$ with the kernel $\frac{x}{|x|^{s+1}}$ where $x\in \R^d\setminus \{0\}$.  The positive lower density assumption was later removed by Tolsa \cite{To4}, who introduced a very novel variation of a scheme of Leg\'er \cite{L} (which in turn has its origins in the work of David-Semmes \cite{DS}).\\

Villa \cite{Vil} recently extended the results of \cite{M} to perturbations of the Cauchy kernel, and it would be interesting to understand whether those results remain valid without the assumption of positive lower density.  The Huovinen kernel does not fall within this perturbative theory, but our analysis does not appear to apply to perturbations in the generality that they are considered in \cite{Vil}.\\

Broadly speaking, we proceed by adapting the scheme implemented by in Tolsa \cite{To4}, but doing so required overcoming a basic difficulty.  A measure $\mu$ for which the Cauchy (or Riesz) transform exists in principal value enjoy some `local flattening' properties\footnote{by this we mean that on a small scale either there is very little measure, or the support of the measure is close to a line/plane of appropriate dimension} on account of the fact that the only \emph{symmetric measures} associated to these kernels with suitable growth are (the Hausdorff measures of) planes.  However, there are symmetric measures associated to the Huovinen kernel which are not flat -- the \emph{spike measures} -- which will appear often in our analysis.   It appears that all variants of the Leg\'er scheme (for instance in \cite{L, CMPT, To5, JTV}) have relied on this local flattening property in one way or another.  In this paper we circumvent these difficulties with a novel decomposition of a measure involving a modified density, relying significantly on our previous papers \cite{JM1,JM2}, which we recall in the next sections.

\subsection{A first necessary condition for existence of principal value:  Small local action and transportation coefficients}  In the paper \cite{JM1} we studied the geometric consequences of a weaker notion than existence of principal value called \emph{small local action}.  

For a homogeneous Calder\'on-Zygmund operator, one can characterize the small local action property geometrically in terms of the transportation distance to the class of symmetric measures associated to the kernel (Theorem 1.1 of \cite{JM1}, building upon work of Mattila \cite{M, M1}).  We do not define these terms here, but rather state what it means for the Huovinen transform.

\begin{definition}\label{spikes}A $k$-spike measure associated to a line $D \in \G_0$ (i.e. going through $0$) and the vertex $z \in \C$ is a measure of the form, for some $c>0$,
$$ \nu_{m,D,z}=c\sum_{n=0}^{m-1} \mathcal{H}_{e^{\pi in/m}D+z},
$$
where $m$ divides $k$ (henceforth $m\mid k$).  We set $\mathcal{S}_k$ to be the collection of all such spike measures over $D\in \G_0$, $z\in \C$.
\end{definition}
Fix the Lipschitz continuous function $\varphi$ that satisfies $\varphi\equiv 1$ on $[0,3)$,  $\|\varphi\|_{Lip} = 1$, and $\supp(\varphi) \subset [0,4)$.
\begin{definition}\label{alphanumbers} Given a locally finite Borel measure $\mu$, $z \in \C$, and $r>0$, we define the transportation distance as
$$ \kalpha_{\mu}(B(z,r)) = \inf_{\substack{\nu \in \mathcal{S}_k:\\ z\in \supp(\nu)}} \alpha_{\mu,\nu}(B(z,r)),$$
where, for a Borel measure $\nu$,
$$ \alpha_{\mu,\nu}(B(z,r))=\sup_{\substack{f \in \lip_0(B(z,4r)) \\ \|f\|_{\lip}\leq \frac{1}{r}}}
 \Bigl|\frac{1}{r}\int_{\C} \varphi\Bigl(\frac{|\cdot-z|}{r}\Bigl) f \;
\,d (\mu - c_{\mu,\nu} \nu) \Bigl|,$$
 and with  the normalizing constant $c_{\mu, \nu}$  $$c_{\mu,\nu}\footnote{For convenience, from now on we will suppress the dependence on both location and radius.}=\begin{cases} \int_{\C} \varphi\bigl(\frac{|\cdot-z|}{r}\bigl)d\mu \Bigl[\int_{\C} \varphi\bigl(\frac{|\cdot-z|}{r}\bigl) d\nu\Bigl]^{-1} \text{ if }\int_{\C}\varphi \bigl(\frac{|\cdot-z|}{r}\bigl) d\nu \neq 0\\  0\text{ otherwise}.\end{cases}$$
\end{definition}

The following result is an immediate consequence of Proposition A.1 and Theorem 1.5 of \cite{JM1}, making essential use of the aforementioned work of Mattila \cite{M} and Huovinen \cite{H}.

\begin{thmx}\cite{JM1} \label{huovSLA} Suppose that $\mu$ is a finite Borel measure satisfying (\ref{limsupmeas}) and (\ref{huovlimit}), then
\begin{equation}\label{alphak}\lim_{r\to 0}\kalpha_{\mu}(B(z,r))=0 \text{ for }\mu\text{-a.e. }z\in \C.
\end{equation}
\end{thmx}

This result provides valuable geometric information without which we would not be able to prove Theorem \ref{PVthm}, but the condition (\ref{alphak}) alone does not imply that $\mu$ is rectifiable, even if $k=1$ --  see for instance the examples in Section 5.8 of \cite{P}.

\subsection{A second necessary condition for the existence of principal value:  Operator boundedness}

We set
$$K_k(z) = \frac{z^k}{|z|^{k+1}}, \text{ for }z\in \C\setminus \{0\} .
$$

For a non-atomic Borel measure $\mu$, we say the Huovinen transform associated to $\mu$ is bounded in $L^2(\mu)$ if there exists $C>0$ such that \begin{equation}\label{L2bdd}\sup_{\kap>0}\int_\C\Bigl|\int_{\C\backslash B(z,\kap)}K_k(z-\omega)f(\omega)d\mu(\omega)\Bigl|^2d\mu(z)\leq C\|f\|^2_{L^2(\mu)}
\end{equation}
for every $f\in L^2(\mu)$.

A well-known consequence (see, for instance \cite{Dav2}, page 56) of the $L^2$-boundedness condition (\ref{L2bdd}) is that $\sup_{z\in \C, r>0}\frac{\mu(B(z,r))}{r}<\infty$.

A simple special case of much more general results of Nazarov-Treil-Volberg \cite{NTV2} and Tolsa \cite{To3} is the following theorem, valid for a wide class of Calder\'on-Zygmund operators.

\begin{thmx}\label{NTVthm}\cite{NTV2}  Suppose that $\mu$ is a finite Borel measure satisfying (\ref{huovlimit}) and \begin{equation}\label{finiteupperdens}\limsup_{r\to 0}\frac{\mu(B(z,r))}{r}<\infty \text{ for }\mu\text{-a.e. }z\in \C. \end{equation} For every $\eps>0$ there is a set $E_{\eps}$ and a constant $C = C(\eps)$ such that $\mu(\C\backslash E_{\eps})<\eps$ and the measure $\mu|_{E_{\eps}}$ satisfies the $L^2$-boundedness condition (\ref{L2bdd}).
\end{thmx}

The following Corollary is immediate from Theorem \ref{NTVthm}.

\begin{cor} \label{NTVcor}  Suppose that $\mu$ is a finite Borel measure satisfying (\ref{huovlimit}) and (\ref{finiteupperdens}).  There is a decomposition $\supp(\mu) = F\cup \bigcup_{j=1}^{\infty}E_j$, where $\mu(F)=0$ and (\ref{L2bdd}) holds with $\mu$ replaced by $\mu|_{E_j}$ for a constant $C=C(j)$.
\end{cor}

We conclude that the notion of principal value is (although qualitative) stronger than the $L^2$-boundedness of the operator.  It is actually \emph{significantly stronger}: In \cite{JN} an example was constructed of a (purely unrectifiable) measure for which the Huovinen transform is bounded in $L^2$, but fails to exist in principal value\footnote{This is another instance in which the Huovinen and Cauchy transforms behave very differently, since if $\mu$ is a non-atomic measure for which the Cauchy transform is bounded in $L^2$, then the Cauchy transform exists in principal value (the same result here also holds for the $(d-1)$-Riesz transform in $\R^d$, as can be seen by stringing together the results of \cite{ENV,NToV, NToV2,MV}).}.  Higher dimensional analogues of this example featuring kernels of spherical harmonics have recently been developed by Mateu and Prat \cite{MaPr}.  It would be very challenging to extend Theorem \ref{PVthm} to this higher dimensional setting -- with the primary issue being to understand the structure of the set of symmetric measures associated to these higher dimensional kernels.  Other examples of kernels  for which $L^2$-boundedness does not imply existence of principal value can be found in \cite{CH, Dav}.

We have thus far recorded two necessary conditions for the existence of principal value in  Theorems \ref{huovSLA} and \ref{NTVthm}.  Taken individually, neither condition needs to imply the existence of principal value, but by building on prior work of Mattila-Verdera \cite{MV}, we showed in \cite{JM2} that when combined, these two necessary conditions are indeed sufficient:

\begin{thmx}\cite[Theorem 1.5]{JM2} \label{JM2thm}  Suppose that $\mu$ is a finite non-atomic Borel measure satisfying the transportation coefficient condition (\ref{alphak}), and the $L^2$-boundedness condition (\ref{L2bdd}) holds.  Then the principal value limit (\ref{huovlimit}) exists.
\end{thmx}

\subsection{A revised statement}  We conclude with a revised statement, which is essentially equivalent to Theorem \ref{PVthm}, and which will be our focus:

\begin{thm} \label{mainthm}
Let $\mu$ be a finite non-atomic Borel measure in the complex plane and whose support satisfies $\mathcal{H}^1(\supp(\mu)) < \infty.$
Suppose that the Huovinen transform is bounded in $L^2(\mu)$ and $$\lim_{r \to 0} \kalpha_{\mu}(B(z,r))=0\text{ for }\mu\text{-a.e. }z \in \C.$$ Then $\mu$ is rectifiable.
\end{thm}

This result is only new if $k\geq 3$ (for $k=1$ it is a consequence of \cite{To3}), but we will prove the statement for all odd $k$ (although many of the statements of lemmas are automatically satisfied in the case $k=1$).  In the case $k=1$, imposing the condition $\lim_{r \to 0} \kalpha_{\mu}(B(z,r))=0$ for $\mu$-a.e. $z \in \C$ is unnecessary -- the result still holds if one removes this statement, which is a theorem due to David \cite{Dav3}, see also David-Mattila \cite{DM}.  \emph{However, for $k\geq 3$ the conclusion of rectifiability may fail without the additional assumption on the transportation numbers} (cf. \cite{JN}).

\subsection{An overview of the proof}\label{layout}
As we have already mentioned, the proof of Theorem \ref{mainthm} follows a similar scheme to the one in Tolsa \cite{To4}:
We decompose our measure into different pieces, where an adapted version of the David-L\'eger-Semmes scheme  \cite{L,DS} may be applied construct a Lipschitz graph that approximates our measure. Finally, we revise Tolsa's scheme in order to prove that the Lipschitz graph actually covers a good portion of our measure.

In order to carry both L\'eger's and Tolsa's schemes in a given scale, one needs, besides of course the analytic properties of the singular integral operator, two specific features from the measure:  flatness with respect to lines and nearly maximal density.   A priori, we are only equipped with spike flatness, i.e. our measure is in concentration close to either a line or to a spike. However, spikes allow big oscillations in density, making  harder the search for suitable scales with nearly maximal density.

These issues are mainly bypassed with the decomposition of the measure (see Section \ref{decomposition}) and the development of a modified density (see Sections \ref{Density} and \ref{newdens}). This new density moves us away from the center of the spikes (therefore  it finds for us scales with regular flatness) and helps us to classify the spikes by the density in their rays.

In Sections \ref{approxgraph} and \ref{F1sec} we carry out a variant of the L\'eger construction of an approximating Lipschitz curve, where the transportation coefficients play a central role. 

The necessary geometric toolbox for Sections \ref{Density}--\ref{F1sec}  is developed in Section \ref{transportation}.

Sections \ref{sizeF2} and \ref{CZsmallconstant}  closely follow Tolsa \cite{To4}, and mainly concern the Calder\'on-Zygmund theory required to show that the approximate Lipschitz curve (constructed in Section \ref{approxgraph}) does not rotate too much. \\

\section{Notation and Preliminaries} \label{notation}

In this section we include the basic notation that we will use throughout the paper and include some preliminaries from geometric measure theory that are relevant for the geometric constructions occupying the first half of the paper.  Notation specific for the analytic part of the paper is included in Section \ref{CZsmallconstant}.

\subsection{Notation}

\begin{itemize}
    \item  We shall denote by $C>0$ and $c>0$ respectively large and small constants that may change from line to line. By $A\lesssim B$, we shall mean that $A\leq CB$ for some constant $C>0$.  $A\approx B$ then means that both $A\lesssim B$ and $B\lesssim A$.  By $A\ll B$ we shall mean that $A\leq c_0B$ for some sufficiently small constant $c_0>0$.
    \item Throughout the paper we will only consider locally finite Borel measures and they will simply be referred to as measures.
    \item An interval in $\R$ will be typically denoted by $I$.  Set $I_0=(-1,1)$.
	\item  $B(z,r)$ denotes the open ball centered at $z\in \C$ with radius $r>0$. Given an open ball $B$, we will denote its center by $c(B)$ and its radius by $r(B)$. Given $\Lambda>0$ we denote by $\Lambda B$ the ball with center $c(B)$ and radius $\Lambda r(B)$.
\item  $\G_z$  denotes the collection of 1-dimensional affine linear subspaces of $\C$ going through $z\in \C$.
\item For $E\subset \C$ set
$$\H^1(E) = \sup_{\delta>0}\Big[\inf\Bigl\{\sum_{j=1}^{\infty} 2r_j\,:\, E\subset \bigcup_{j=1}^{\infty} B(x_j, r_j) \text{ and } r_j\leq \delta\Bigl\}\Bigl].
$$
With this normalization, for $L\in \G_z$, $\H^1_{|L}$ coincides with the usual one-dimensional Lebesgue measure on $L$.
\item For a function $f$ defined on an open set $U\subset \C$, define
$$\|f\|_{\Lip(U)} = \sup_{x,y\in U, \, x\neq y}\frac{|f(x)-f(y)|}{|x-y|}.
$$
In the case $U=\C$, we write $\|f\|_{\Lip}$ instead of $\|f\|_{\Lip(\C)}$.
\item For an open set $U\subset \C$, define $\Lip_0(U)$ to be the collection of functions $f$ supported on a compact subset  of $U$ with $$\|f\|_{\Lip(U)}<\infty.$$
	\item We denote by $\supp(\mu)$ the closed support of the measure $\mu$; that is,
	$$ \supp(\mu) = \C \setminus \{ \cup B: B \text{ is an   open ball with }\mu(B)=0 \}.$$        	
    \item $\delta_\mu(B(z,r))=\frac{\mu(B(z,r))}{2r}$ is referred to as the density of $\mu$ at the scale $B(z,r)$. 
    \item We denote by $\Theta^\ast_\mu(z)=\limsup_{r \to 0} \frac{\mu(B(z,r))}{2r}$, the upper density of the measure $\mu$ at the point $z$.    
    \item For $x\in \C$,  write $x=\Re(x)+i\Im(x)$.  Denote by $\pi$ the projection from $\C \to \R$:
    $$ \pi(x)= \Re(x).$$
    
    \item We will use the notation $$\varphi_{z,r}(y)=\varphi\Bigl(\frac{|y-z|}{r}\Bigl), \text{ for } y \in \C.$$
    \item We define the class of functions $\mathcal{F}_{z,r}$ as follows:
      $$\mathcal{F}_{z,r}=\{ f : f \in \Lip_0(B(z,4r)), \|f\|_{\Lip} \leq 1/r\}.$$

          \item Given a ball $B$ and a line $D \in \G_{c(B)}$, it will be convenient to  write $\alpha_{\mu,D}(B)$ instead of $\alpha_{\mu,\mathcal{H}^1_{|D}}(B)$. We will often refer to measures of the form $c\mathcal{H}^1_{|D}$ for some $c>0$ as a line measures. 
 
\end{itemize}
  
 \subsection{Two transportation numbers that will recur throughout the work}   We will mainly work with two transportation numbers (Definition \ref{alphanumbers}).  Recall that $\S_k$ is the set of $k$-spike measures, and so $\S_1$ is the set of line measures in $\C$.  We set

\begin{itemize}
    \item $\kalpha_\mu(B(z,r))$
as the transportation coefficient with respect to spikes and 
 \item $\alpha_\mu(B(z,r))=\alpha^{(1)}_{\mu}(B(z,r))$ as the transportation coefficient with respect to lines. 
 \end{itemize}

\subsection{Basic Operator Notation}

For a kernel function $K:\C\times \C\backslash \{(z,\omega): z=\omega\}\to \C$ such that $|K(z,\omega)(z-\omega)|$ extends to a bounded function on $\C\times \C$, we set
$$P.V. \int_{\C}K(z,\omega) f(\omega)d\mu(\omega) = \lim_{r\to 0}\int_{|z-\omega|>r}K(z,\omega) f(\omega)d\mu(\omega)
$$
provided that the right hand side exists.  We say that $K$ forms a principal value operator on $L^p(\mu)$ ($1<p<\infty$) if there is a constant $C>0$
\begin{equation}\label{PVOp}\int_{\C}\Bigl|P.V. \int_{\C}K(z,\omega) f(\omega)d\mu(\omega)\Bigl|^p d\mu(z)\leq C\|f\|_{L^p(\mu)}^p
\end{equation}
 for all $f\in L^p(\mu)$.  We call the least constant $C$ such that (\ref{PVOp}) holds as the \emph{principal value operator norm}.

It will prove very useful to define operators with a smoother cut-off.  Define a function $\Psi:[0,\infty)\to [0, \infty)$ such that $\Psi$ is non-decreasing, $\Psi(t)\equiv 0$ on $[0, 1/2]$ and $\Psi(t)=1$ for $t\geq 1$, and $\|\Psi''\|_{\infty}\lesssim 1$. Put, for $r>0$ and any measure $\nu$
$$\HT_r \nu(z) = \int_{\C}\Psi\Bigl(\frac{|z-\omega|}{r}\Bigl) K_k(z-\omega) d\nu(\omega),
$$
$$\HT^\perp_r \nu(x) = \int_{\C}\Psi\Bigl(\frac{|z-\omega|}{r}\Bigl) K^\perp_k(z-\omega)d\nu(\omega),
$$
where $K^\perp_k(z)=\displaystyle \frac{\Im (z^k)}{|z|^{k+1}}$ for $z \in \mathbb{C}\setminus\{0\}$, 

$$\HT_{r_1,r_2} \nu (x)=\begin{cases}\HT_{r_1} \nu(x)-\HT_{r_2} \nu(x) \text{ if }r_1<r_2\\0 \text{ if }r_2\geq r_1,\end{cases}$$
and
$$T^\perp_{r_1,r_2} \nu (x)=\begin{cases}\HT^\perp_{r_1} \nu(x)-\HT^\perp_{r_2} \nu(x) \text{ if }r_1<r_2\\0 \text{ if }r_2\geq r_1.\end{cases}$$

\begin{lemma}\label{smoothnochange} Suppose $\mu$ is a locally finite Borel measure.
\begin{enumerate} 
\item If the principal value limit (\ref{huovlimit}) exists at a given point $z\in \C$, then $\lim_{r\to 0}\HT_r(\mu)(z)$ exists and is equal to the same limit.
\item If the Huovinen transform is bounded in $L^2(\mu)$  (in the sense that (\ref{L2bdd}) holds), then there is a constant $C$ such that
\begin{equation}\label{HTOp}\|\sup_{r>0}|\HT_r(f\mu)|\|_{L^2(\mu)}^2 \leq C\|f\|_{L^2(\mu)}^2 \text{ for every }f\in L^2(\mu)
\end{equation}
\end{enumerate}
\end{lemma}

The proof of (1) is by direct calculation, while (2) is standard Calder\'on-Zygmund theory: one estimates the difference between the smooth and rough cut-off by a suitable maximal function, and applies a Cotlar type lemma to bound the maximal singular integral (see \cite{To5}, Chapter 2). \\

In the event that there is a constant $C$ such that (\ref{HTOp}) holds, we denote the least such constant by $\|\HT_{\mu}\|_{L^2(\mu), L^2(\mu)}$.\\

We warn the reader here that, even if the associated principal value operator exists and is bounded in $L^2(\mu)$, then $\|\HT_{\mu}\|_{L^2(\mu), L^2(\mu)}$ need not be comparable with the principal value operator norm.

\section{Transportation coefficients tool box}\label{transportation}

Now we proceed to record a series of estimates regarding the transportation coefficients that will be used throughout the paper.

Throughout this section, $\nu$ will denote a locally finite Borel measure.

\begin{lemma}\label{basicalphasmall}
Let $\gamma >0$ and suppose  $s\in (0, r)$,  $B(z,s)\subset B(x, 3r)$,  $\dist(z,\supp(\nu))\geq 2s$, and
$$\alpha_{\mu, \nu}(B(x,r))\leq \gamma \delta_{\mu}(B(x,r)).$$
Then

$$\delta_{\mu}(B(z,s))\leq \gamma\Bigl(\frac{r}{s}\Bigl)^2 \delta_{\mu}(B(x,r)).
$$
\end{lemma}

\begin{proof} Choose $f\equiv 1 $ on $B(z,s)$ with $\supp(f)\subset B(z, 2s)$ and $\|f\|_{\Lip}\leq \frac{1}{s}$.  Then $\frac{s}{r}f\in \mathcal{F}_{x,r}$.  Since $\alpha_{\mu, \nu}(B(x,r))\leq \gamma \delta_{\mu}(B(x,r))$, but $\supp(f)\cap \supp(\nu)=\varnothing$,
$$\frac{1}{r}\cdot \frac{s}{r}\cdot\mu(B(z,s))\leq \gamma\delta_{\mu}(B(x,r)),
$$
and the result follows.
\end{proof}

\begin{lemma}\label{almostmono} Let $\gamma >0$ and suppose $s\in (0,r/2)$, $B(z,3s)\subset B(x,3r)$, and
$$\alpha_{\mu, \nu}(B(x,r))\leq \gamma \delta_{\mu}(B(x,r)).
$$
Then
$$\alpha_{\mu, \nu}(B(z,s))\lesssim \gamma\Bigl(\frac{r}{s}\Bigl)^2\delta_{\mu}(B(x,r)).
$$
\end{lemma}
\begin{proof} Without loss of generality, suppose $x=0$, $r=1$ and $\mu(B(0,1))=1$.  Insofar as $B(z,3s)\subset B(0,3)$ and $s<1/2$, 
$\supp(\varphi_{z,s})\subset \{\varphi_{0,1}\geq \tfrac{1}{2}\}$ and so the function $g = \frac{\varphi_{z,s}}{\varphi_{0,1}}\in \Lip_0(B(0, 4))$ with $\|g\|_{\Lip}\lesssim \frac{1}{s}$.
For $f\in \mathcal{F}_{z,s}$, the function $\frac{s}{C}f \cdot g \in \mathcal{F}_{0,1}$ for a suitable constant $C>0$, so testing the condition $\alpha_{\mu, \nu}(B(0,1))\leq \gamma$ yields that
$$\Bigl|\int f \varphi_{z,s}d\left(\mu - \frac{\int \varphi_{0,1} \,d\mu}{\int \varphi_{0,1} \,d\nu}\nu\right)\Bigl|\lesssim \gamma\frac{1}{s}.
$$
On the other hand, testing the condition $\alpha_{\mu, \nu}(B(0,1))\leq \gamma$ with the function $ \frac{s}{C}g$ yields
$$\Bigl|\int \varphi_{z,s}d\mu - \frac{\int \varphi_{0,1} \,d\mu}{\int \varphi_{0,1} \,d\nu}\int \varphi_{z,s} d\nu\Bigl|\lesssim \gamma\cdot \frac{1}{s}.
$$
The required estimate is now obtained by combining these two inequalities.
\end{proof}

\begin{lemma}[Continuity of transportation coefficients]\label{continuity}
Given a sequence $\{(x_j,r_j)\}_{j \geq 0} \in \C \times (0,\infty)$ satisfying that $x_j \to x_0 $ and $r_j \to r_0$, we have the following:
\begin{enumerate}
    \item $\alpha_\mu(B(x_j,r_j)) \to \alpha_\mu(B(x_0,r_0)) $.
    \item Moreover, given a sequence $D_j \in \G_{x_j}$ for all $j \geq 0$ satisfying $\angle(D_j,D_0) \to 0$, then $\alpha_{\mu,D_j}(B(x_j,r_j)) \to \alpha_{\mu,D}(B(x_0,r_0)).$
\end{enumerate}
\end{lemma}
We postpone the proof to the appendix.

\section{Density ratio}\label{Density}

For a non-zero measure we set
\begin{equation}\label{densratio}D_{\nu} = \sup_{\substack{r,s>0\\x,z\in \supp(\nu)}}\frac{\delta_{\nu}(B(x,r))}{\delta_{\nu}(B(z,s))}.
\end{equation}

Observe that
\begin{itemize}
\item for any non-zero measure $\nu \not \equiv 0$, $D_{\nu}\geq 1$,
\item if $\nu$ is a line measure, then $D_{\nu}=1$, and
\item if $\nu \in \mathcal{S}_k$, then $D_{\nu} \leq k$.
\end{itemize}

\begin{lemma}\label{moveoffsupport}
Given a measure $\nu$ and  $x\in \C$, $$\delta_{\nu}(B(x,r))\leq 3D_{\nu}\cdot \delta_{\nu}(B(z,s))\text{ for every }r, s>0\text{ and }z\in \supp(\nu).$$  Moreover, if $\nu$ is a line measure, then
$$\delta_{\nu}(B(x,r))\leq \delta_{\nu}(B(z,s))\text{ for every }r,s>0\text{ and }z\in \supp(\nu).
$$
\end{lemma}

\begin{proof}
If $\nu(B(x,r))=0$ then there is nothing to prove.  Otherwise $r>\dist(x, \supp(\nu))$, and fix $x_{\nu}\in \supp(\nu)$ to be the closest point to $x$.  The first statement follows from noticing that $B(x,r)\subset B(x_{\nu}, 3r)$.  For the second statement, merely observe that if $\nu$ is a line measure, then $B(x,r)\cap\supp(\nu)\subset B(x_{\nu},r)\cap \supp(\nu)$.
\end{proof}

\begin{lemma}\label{densitycomparison} Let $\gamma > 0$ and suppose  $s\in (0,r]$, $B(z,s)\subset B(x,3r)$, and 
$$\alpha_{\mu, \nu}(B(x,r))\leq \gamma \delta_{\mu}(B(x,r)),
$$
for some measure $\nu$ satisfying that $x \in \supp(\nu)$.
 Then,
\begin{enumerate}
\item[(1)]  if $\gamma<\frac{1}{9}(s/r)^2$, one has $$\delta_{\mu}(B(z,s))\leq 3D_{\nu}\Bigl(1+8\sqrt{\gamma}\cdot\frac{r}{s}\Bigl)\delta_{\mu}(B(x,r)),$$
and moreover, if $\nu$ is a line measure,
$$\delta_{\mu}(B(z,s))\leq \Bigl(1+8\sqrt{\gamma}\cdot\frac{r}{s}\Bigl)\delta_{\mu}(B(x,r)).$$
\item[(2)]  If $\gamma<\frac{1}{9D_{\nu}}(s/r)^2$, and in addition, $z\in \supp(\nu)$, then $$\delta_{\mu}(B(z,s))\geq D_{\nu}^{-1}\Bigl(1-8\sqrt{D_{\nu}\gamma}\cdot \frac{r}{s}\Bigl)\delta_{\mu}(B(x,r)).$$
\end{enumerate}

\end{lemma}

\begin{proof}
The statements are both trivial if $D_{\nu}=+\infty$, so we assume otherwise. Additionally, if $\delta_{\mu}(B(x,r))=0$, then $\mu=\nu$ in $B(x, r)$, but $x\in \supp(\nu)$ so $\mu(B(x,r))=\nu(B(x,r))>0$, a contradiction.  Therefore $\delta_{\mu}(B(x,r))>0$.

Now, without loss of generality we set $x=0$, $r=1$, and $\mu(B(0,1))=1$. Then $0\in \supp(\nu)$.  We first prove $(1)$. Fix $\eta\in (0,1/3)$.  Pick two bump functions $f_1$ and $f_2$ satisfying 
\begin{itemize}
    \item $f_1 \equiv 1 $ on $B(z,s)$, $f_1 \equiv 0$ outside $B(z,(1+\eta)s)$, \\
    $ 0 \leq f_1 \leq 1$, and $\|f_1\|_{\Lip} \leq 1/(\eta s)$.
    \item $f_2 \equiv 1$ on $B(0,1-\eta)$, $f_2 \equiv 0$ outside $B(0,1)$, \\
    $ 0 \leq f_2 \leq 1$, and $\|f_2\|_{\Lip} \leq 1/\eta$.
\end{itemize}
On one hand, observe that $\eta s f_1 \in \mathcal{F}_{0,1}$ and therefore testing the condition $\alpha_{\mu,\nu}(B(0,1))\leq \gamma$ with $\eta a f_1$ yields 
\begin{equation}\label{cmuvulow}
\frac{\mu(B(z,s))}{s} \leq \frac{\nu(B(z,(1+\eta)s))}{s}\int_{\C} \varphi d\mu \Bigl[\int_{\C} \varphi d\nu\Bigl]^{-1} + \frac{\gamma}{\eta s^2}.
\end{equation}
On the other hand, we notice that $\eta f_2 \in \mathcal{F}_{0,1}$, and hence by analogous reasoning, 
\begin{equation}\label{cmuvuup}1=\mu(B(0,1)) \geq \nu(B(0,1-\eta))\int_{\C} \varphi d\mu \Bigl[\int_{\C} \varphi d\nu\Bigl]^{-1}-\frac{\gamma}{\eta}.\end{equation}
Set (cf. Lemma \ref{moveoffsupport}) $$\kap = \begin{cases} 1 \text{ if }\nu \text{ is a line measure}\\
3\text{ otherwise}\end{cases}.$$Bringing (\ref{cmuvulow}) and (\ref{cmuvuup}) together, we obtain
\begin{align*}
    \frac{\mu(B(z,s))}{s}& \leq \Bigl(1 +\frac{\gamma}{\eta}\Bigl)\frac{\nu(B(z,(1+\eta)s))}{s\cdot \nu(B(0,1-\eta))} +\frac{\gamma}{\eta s^2}\\
    & =  \Bigl(1 +\frac{\gamma}{\eta}\Bigl)\frac{1+\eta}{1-\eta}\frac{\delta_\nu(B(z,(1+\eta)s))}{\delta_\nu(B(0,1-\eta))} +\frac{\gamma}{\eta s^2}\\
    & \leq \Bigl(1 +\frac{\gamma}{\eta}\Bigl)\frac{1+\eta}{1-\eta}\kap D_{\nu}+\frac{\gamma}{\eta s^2}\\
    &\leq \Bigl(\frac{1+\eta}{1-\eta}+ \frac{3\gamma}{\eta s^2}\Bigl)\kap D_{\nu},
\end{align*}
where in the final inequality we have used the facts that $\frac{1+\eta}{1-\eta}\leq 2$, $\kap\geq 1$, $D_{\nu}\geq 1$ and $s\leq 1$.  Put
$\eta^{-1} = s\sqrt{\frac{2}{3\gamma}}+1$ so that
then
$$\frac{1+\eta}{1-\eta}+ \frac{3\gamma}{\eta s^2} = 1+2\sqrt{6}\frac{\sqrt{\gamma}}{s}+\frac{3\gamma}{s^2}\leq 1+(1+2\sqrt{6})\frac{\sqrt{\gamma}}{s},
$$
and (1) follows.

The proof of $(2)$ follows an entirely analogous line of reasoning.  Again fix $\eta\in (0, 1/3)$.   First notice that testing $\alpha_{\mu,\nu}(B(0,1))\leq \gamma$ with suitable test functions yields
\begin{align*}
    \frac{\mu(B(z,s))}{s} & \geq \Bigl(1-\frac{\gamma}{\eta}\Bigl)\frac{\nu(B(z,(1-\eta) s))}{s\nu(B(0,1+\eta))}-\frac{\gamma}{\eta s^2}.
\end{align*}
Next, observe that due to the fact that $z \in \supp(\nu)$,
$$\frac{\nu(B(z,(1-\eta)s)}{s\nu(B(0,1+\eta))} \geq \frac{1-\eta}{1+\eta}D_{\nu}^{-1},$$
and so
$$\delta_{\mu}(B(z,s))\geq \Bigl( \frac{1-\eta}{1+\eta}-\frac{2\gamma D_{\nu}}{\eta s^2}\Bigl) D_{\nu}^{-1}.
$$
(Here we are using that $\frac{1-\eta}{1+\eta}\leq 1$, and $D_{\nu}\geq 1$.)  Choosing  $\eta^{-1} = s\sqrt{\frac{2}{3\gamma D_{\nu}}}-1$ we complete the proof of part (2) with some elementary manipulations.
\end{proof}

\begin{lemma}\label{linesdontmove} Let $\delta, \gamma >0$ with $\delta \leq 1$ and suppose $s\in (0,r]$, $B(z,s)\subset B(x,2r)$, $\delta_{\mu}(B(z,s))\geq \delta \cdot \delta_{\mu}(B(x,r))$, and $$\alpha_{\mu, \sigma}(B(x,r))\leq \gamma \delta_{\mu}(B(x,r)), \; \alpha_{\mu, \nu}(B(z,s))\leq \gamma \delta_{\mu}(B(z,s)),
$$
where  $\nu$ and $\sigma$ are measures such that $x\in \supp(\sigma)$ and $z\in \supp(\nu)$.  Then for every $y\in \supp(\nu)\cap B(z,s),$ $$\min\Bigl\{s,\dist(y,\supp(\sigma))\Bigl\}\lesssim \sqrt{\frac{\gamma \cdot D_{\nu}}{\delta}}\cdot r .$$
\end{lemma}

\begin{proof}Suppose $x=0$, $r=1$ and $\mu(B(0,1))=1$.  Fix $y\in B(z,s)\cap \supp(\nu)$, and set $t=\min\{s,\frac{1}{2}\cdot \dist(y, \supp(\sigma))\}$.  We may assume that $t\geq 16\sqrt{D_{\nu}\gamma}\cdot s$ as otherwise the claimed estimate is clearly true.

Under this assumption on $t$, part (2) of Lemma \ref{densitycomparison} ensures that
$$\delta_{\mu}(B(y,t))\geq D_{\nu}^{-1}\Bigl(1-8\sqrt{D_{\nu}\gamma}\frac{s}{t}\Bigl)\delta_{\mu}(B(z,s))\geq \frac{1}{2D_{\nu}}\delta_{\mu}(B(z,s))\geq \frac{\delta}{2D_{\nu}}.
$$

On the other hand, by construction $t\leq 1$, so $B(y,t)\subset B(0,3)$ and by Lemma \ref{basicalphasmall},
$$\delta_{\mu}(B(y,t))\leq \frac{\gamma}{t^2}\mu(B(0,1))=\frac{\gamma}{t^2}.
$$
Joining these two chains of inequalities together, we obtain $$\delta \leq 2D_{\nu}\frac{\gamma}{t^2},$$
and this yields the desired upper bound on $t$.
\end{proof}

The following Corollary is an immediate consequence of this lemma in the case when $\nu$ and $\sigma$ are line measures, but it will be used very often in what follows so we state it separately.

\begin{cor} \label{linesdontmovecor} Let $\gamma >0$ and  $\delta \in (0,1]$. Suppose $s\in (0,r]$, $B(z,s)\subset B(x,2r)$,  $\delta_{\mu}(B(z,s))\geq \delta \cdot \delta_{\mu}(B(x,r))$ and there exist  $D\in \G_x$ and $D'\in \G_z$ such that
$$\alpha_{\mu, D}(B(x,r))\leq \gamma \delta_{\mu}(B(x,r)), \; \quad \alpha_{\mu, D'}(B(z,s))\leq \gamma \delta_{\mu}(B(z,s)).
$$   Then
$$\min\Bigl\{s,\dist(y, D)\Bigl\}\lesssim \sqrt{\frac{\gamma}{\delta}}\cdot r \text{ for every }y\in D'\cap B(z,s),$$
and therefore
$$\angle(D,D')\lesssim \sqrt{\frac{\gamma}{\delta}}\cdot\frac{r}{s}.$$
\end{cor}

The next lemma will play a crucial role in the stopping time argument.

\begin{lemma}\label{spikeflatter} Fix $\delta\in (0,1]$ and  $\gamma >0$ with $\gamma \ll \delta$. There is a constant $C>0$ such that the following holds: \\ Suppose that $B(z,4s)\subset B(x, 2r)$ where $s\in [C\sqrt{\frac{\gamma}{\delta}}r, \tfrac{1}{4}r]$,  and additionally
\begin{itemize}
\item $\alpha_{\mu, D}(B(x,r))\leq \gamma \delta_{\mu}(B(x,r))$,
\item $\kalpha_{\mu}(B(z,4s))\leq \gamma^2\delta_{\mu}(B(x,r))$, and 
\item $\delta_{\mu}(B(z,s))\geq \delta\cdot  \delta_{\mu}(B(x,r))$.
\end{itemize}
Then there exists $D' \in \G_z$ such that
$$\alpha_{\mu, D'}(B(z,s))\lesssim  \frac{\gamma^2}{\delta} \delta_{\mu}(B(z,s)) \text{ and } \angle(D',D)\lesssim \sqrt{\frac{\gamma}{\delta}}\frac{r}{s}.
$$

\end{lemma}

\begin{proof} Fix $q\geq 1$ to be chosen momentarily, and suppose that $s\in[ q \sqrt{\frac{\gamma}{\delta}}r, \tfrac{1}{4}r]$. Insofar as $\delta_{\mu}(B(z,s))\geq \delta\cdot  \delta_{\mu}(B(x,r))$, there is a spike measure $\nu$, with $z\in \supp(\nu)$, satisfying
$$\alpha_{\mu,\nu}(B(z,4s)) \leq 2\frac{\gamma^2}{\delta}\delta_{\mu}(B(z,s)).$$ 
There is a line $D'$ in $\supp(\nu)$ that contains $z$.    If $\nu_{|B(z, 4s)} = \H^1_{|D'\cap B(z, 4s)}$, then certainly the desired inequality holds (see Lemma \ref{almostmono}).  So suppose not and therefore there is another line $D''$ in the support of $\nu$ which intersects $B(z, 4s)$.    But now since $\angle(D', D'')\geq \pi/k$, there must be a point $y\in B(z, 4s)\cap (D'\cup D'')\subset B(z,4s)\cap \supp(\nu)$ that is at a distance $\gtrsim s$ from the line $D$.  On the other hand,  $B(z, 4s)\subset B(x,2r)$ and $y\in B(z, 4s)$ so Lemma \ref{linesdontmove}\footnote{Applied with $\sigma \mapsto \mathcal{H}^1|_D$, $\nu \mapsto \nu$, and $s$ replaced by $4s$.}  ensures that (recall $D_{\nu}\leq k$)
$$\min(s,\dist(y, D))\lesssim \sqrt{\frac{\gamma}{\delta}}r\lesssim \frac{s}{q}.
$$
But $\dist(y, D)\gtrsim r$, so we reach a contradiction if $q$ is large enough.  Therefore, setting $C=q$, we must have that $\alpha_{\mu, D'}(B(z,s))\lesssim \frac{\gamma^2}{\delta}\delta_{\mu}(B(z,s))$. 

Since $\gamma\ll \delta$, the second assertion of the lemma now follows from Corollary \ref{linesdontmovecor}.\end{proof}

\section{Navigating through spikes: a modified density} \label{newdens}
 
 We introduce a density that enables one to find a flat piece of a measure $\mu$ given that $\mu$ is close to a spike in transportation distance. 
 
 For $\nu\in \mathcal{S}_k \setminus \{ \nu \not\equiv 0\}$, set
 $$\lambda_{\nu} =\!\!\!\inf_{\substack{x\in \supp(\nu)\\ r>0}}\frac{1}{r}\sup\left\{t\in (0, r):\!\!\!\begin{array}{l} \text{ there are }B(z,t)\subset B(x,r),\, z\in \supp(\nu),\\
 \text{ a line }D \in \mathcal{G}_z, \text{ and  }c>0,\text{ such that}\\ \; \nu_{|B(z, 4\cdot \J t)} = c\mathcal{H}^1_{|D\cap (B(z, 4\cdot \J t))}\end{array}\!\!\right\}.
 $$
 and
 $$\lambda_k = \inf_{\nu\in \mathcal{S}_k, \nu \not\equiv 0} \lambda_{\nu}.
 $$
 We will often use the simple observation that  $\lambda_k\gtrsim 1$.

Now recall the density ratio (\ref{densratio}).  We define
$$D_k = \sup_{\nu\in \mathcal{S}_k, \,\nu \not\equiv 0}D_{\nu}.
$$
Observe that $1\leq D_k \leq k \lesssim 1$.\\
 
    Fix $\eps\ll 1$.   For $x\in \C$ and $r>0$,  set
    \begin{align*}
      S_{x,r}(\eps)=\left\{B:\!\! \begin{array}{l} B \text{ a ball},B\subset B(x,r), \delta_\mu(B) \geq \frac{1}{2D_k} \delta_\mu(B(x,r)), \\ r(B) \geq \frac{\lambda_k}{2} r, \text{ and }   \alpha_\mu(\J B) \leq \eps \delta_\mu(B) \end{array}\right\}.
  \end{align*}

 We then define the modified density
 $$\wt{\delta}_{\mu,\eps}(B(x,r)) = \begin{cases} \;\inf_{B\in S_{x,r}(\eps)}\delta_{\mu}(B) \text{ if }S_{x,r}(\varepsilon)\neq \varnothing\\
 \;0\text{ otherwise}.\end{cases}
 $$
 
We will usually just drop the subscript $\eps$, and write $\wt{\delta}_{\mu}(B(x,r))$ instead of $\wt{\delta}_{\mu,\eps}(B(x,r))$.  Observe that we have, for any ball $B(x,r)$ and $B\in S_{x,r}(\eps)$,
\begin{equation}\label{trivupperbound}
\delta_{\mu}(B)\leq \frac{2}{\lambda_k}\delta_{\mu}(B(x,r)), \text{ and so } \wt{\delta}_{\mu}(B(x,r))\leq \frac{2}{\lambda_k} \delta_{\mu}(B(x,r)).
\end{equation}

  \begin{lemma} \label{split}
Let $x \in \C$ and $r>0$ be such that \begin{equation}\label{ksmallassump}\kalpha_{\mu}(B(x, \J  r)) \ll \eps \delta_\mu(B(x,r)).\end{equation} Then we have that $$S_{x,r}(\eps)\neq \emptyset$$ and  
\begin{equation}\label{modifiedequiv}
 \frac{1}{C_k}\delta_\mu(B(x,r)) \leq \wt{\delta}_\mu(B(x,r)) \leq C_k \delta_\mu(B(x,r))
 \end{equation}
where $C_k  = \max \{2D_k, 2/\lambda_k\}$.\end{lemma}

\begin{proof}Without loss of generality, we assume that $x=0$, $r=1$ and $\mu(B(0,1))=1$. Choose $\nu \in \mathcal{S}_k$ such that $\alpha_{\mu, \nu}(B(0,30))\leq \varkappa\cdot\eps$ with $\varkappa\ll 1$.

First, we note that it suffices to verify that $S_{0,1}(\eps) \neq \emptyset$. Indeed, if this is the case, then the lower bound in (\ref{modifiedequiv}) is given by the definition of $S_{0,1}(\varepsilon)$, while the upper bound follows from (\ref{trivupperbound}).

Now we proceed to prove that $S_{0,1}(\eps)\neq \emptyset$.  If the measure $\nu$ is a line, then the scale $B(0,1)$ itself belongs to $S_{0,1}(\eps)$. 
 
  If $\nu$ is a spike, i.e. $\nu \in \mathcal{S}_k \setminus \mathcal{S}_1$, then by using the definition of $\lambda_\nu$, we can find $z \in \supp(\nu) \cap B(0,1)$ and $s>0$ satisfying $ s \geq \frac{1}{2}\lambda_\nu\geq \frac{1}{2}\lambda_k\gtrsim 1$, $B(z,s) \subset B(0,1)$, and such that $\nu_{|B(z, 4\cdot \J s)}=c\mathcal{H}^1_{|L\cap B(z, 4\cdot \J s)},$ for some line segment $L$ and $c>0$.
  
   Now using Part 2 of Lemma \ref{densitycomparison}, we have that 
   \begin{align*}
   \delta_\mu(B(z,s))& \geq \frac{1}{2D_{\nu}}\delta_\mu(B(0,1))\geq \frac{1}{2D_k},
   \end{align*}
   and so $ \delta_\mu(B(z,s))\approx 1$.  On the other hand, $B(z, \J s)\subset B(0,30)$ and $s\gtrsim 1$, so Lemma \ref{almostmono} ensures that$$\alpha_\mu(B(z,\J s)) \lesssim \varkappa \eps\lesssim  \varkappa \eps  \delta_\mu(B(z,s)) < \eps \delta_\mu(B(z,s)).$$
    This proves that $S_{0,1}(\varepsilon) \neq \emptyset$.
  \end{proof}
 
 Our last preparatory lemma is an essential ingredient to push through an analogue of Tolsa's scheme.  It says, roughly, that for flat scales, control of $\wt{\delta}_{\mu}$ prevents the density $\delta_{\mu}$ from being too large.

  \begin{lemma} \label{densG} 
 Fix $\theta\in (0,1)$ and $0< \eps \ll \theta$.  Suppose $\alpha_{\mu}(B(x, \J r)) \leq \eps$ and $\wt{\delta}_\mu(B(x,r)) \leq 1 + \theta$.  Then 
 for every $B'\subset B(x,\J r)$ satisfying $r(B')\geq \frac{1}{200}\eps^{1/4}r$ we have that
 $$\delta_\mu(B')\leq 1 + \theta + C\eps^{1/8}.$$
\end{lemma}

\begin{proof}
Without loss of generality, set $x=0$ and $r=1$.  Fix $D\in \D_0$ with $\alpha_{\mu,D}(B(0, \J )) \leq (1+\theta)\eps$.  Assume there exists a ball $B' \subset B(0,\J)$ satisfying $r(B') \geq \frac{1}{200}\varepsilon^{1/4}$ and $\delta_\mu(B')\geq 1 + \theta+ L \eps^{1/8}$ for a large constant $L$. By monotonicity of the measure, this ensures that $\delta_{\mu}(B(0,\J))\gtrsim \eps^{1/4}$.  Therefore,
 $$\alpha_{\mu,D}(B(0,\J)) \lesssim \eps^{3/4}\delta_{\mu}(B(0,\J)),
 $$
 and since $\eps^{3/4}\ll (r(B'))^2$, using parts (1) and (2) of Lemma  \ref{densitycomparison} (in that order) results in the following chain of inequalities:
 $$1 + \theta + L \eps^{1/8} \leq \delta_\mu(B')\leq  (1+C\eps^{1/8})\delta_{\mu}(B(0,\J))\leq (1+C\eps^{1/8})\delta_{\mu}(B(0,1)).$$ 
Insofar as $\eps\ll 1$,  if $L$ is large enough then 
\begin{equation*}
     \delta_\mu(B(0,1))\geq 1+\theta + \frac{L}{2}\eps^{1/8}.
\end{equation*}
We notice that the previous trivially implies that 
$$\alpha_{\mu,D}(B(0,\J))\leq \eps \delta_{\mu}(B(0,1)),$$
and so in particular $S_{0,1}(\eps) \neq \emptyset$.

 Now, insofar as $\wt{\delta}_\mu(B(0,1))\leq 1+\theta$, we can find a ball $\wt{B}=B(z,s)\in S_{0,1}(\eps)$ with $$\frac{1}{2D_k}\leq \delta_\mu(\wt{B})\leq 1+\theta+\eps^2 \text{ and }\alpha_\mu(\J \wt{B})\leq \eps \delta_\mu(\wt{B})\lesssim \eps \delta_{\mu}(\J\wt{B}).$$ Since $s\approx 1$, Lemma \ref{linesdontmove}\footnote{applied with $r$ and $s$ replaced by $30r$ and $30s$ respectively, $\sigma = \mathcal{H}^1|_D$, $\nu$ equal a  line measure with $\alpha_{\mu, \nu}(30\wt{B})\lesssim \eps\delta_{\mu}(30\wt{B})$,  and $\delta$ replaced by $\frac{1}{2D_k}\gtrsim 1$} ensures that $d(z,D)\lesssim \eps^{1/2}$.  Therefore, we can inscribe in $\wt{B}$ a ball $\wh{B}$ centred on $D$ of radius $(1-C\sqrt{\eps})s$, and so $\delta_{\mu}(\wt{B})\geq (1-C\sqrt{\eps})\delta_{\mu}(\wh{B})$. But now part (2) of Lemma \ref{densitycomparison} ensures that $$\delta_\mu(\wh{B}) \geq (1-C\sqrt{\eps})\delta_{\mu}(B(0,1))\geq (1-C\sqrt{\eps})\Big(1+\theta + \frac{L}{2}\eps^{1/8}\Big).$$
Finally, for large enough $L$,  $$\delta_\mu (\wt{B}) \geq 1+\theta +\frac{L}{4}\eps^{1/8},$$
  reaching our desired contradiction.
\end{proof}

  \section{The Main Lemma and the proof of Theorem \ref{mainthm} }
  \label{decomposition}
  Now we are ready to state the Main Lemma.  
  \begin{mainlemma}\label{mainlemma}
Fix $M>1$, $\eps\in (0,1)$, and $\theta\in (0,1)$.   Let $B_0=B(x_0,r_0)$ be an open ball and $F$ a compact subset with $F \subset 10B_0$ satisfying
   \begin{itemize}
       \item[(a)] $\delta_\mu(B_0)=1$, $\alpha_\mu(\J B_0) \leq  \varepsilon$, and $\mu(10 B_0\setminus F) \leq  \varepsilon r_0,$
       \item [(b)]  $\wt{\delta}_\mu (B(x,r))\leq 1+ \theta^2 $ for all $x \in F$ and $r\in (0, 90 r_0)$,
       \item [(c)] $\kalpha_{\mu}(B(x,r)) < \varepsilon^2$ for every ball $B(x,r)$  where $x\in F$ and $ r \in (0, 600 r_0)$,
       \item [(d)] $\|\HT_\mu\|_{L^2(\mu),L^2(\mu)} \leq M$,
       \item [(e)] $|\HT_{r_1,r_2}(\mu)(x)|< \varepsilon$ for all $x\in F$ and $ r_1,r_2 \in (0,90r_0)$.

   \end{itemize}
   There exists an absolute constant $c_0>0$ such that if $\theta$ is chosen small enough depending on $M$, and $\eps$ is chosen small enough in terms of $\theta$ and $M$, then there Lipschitz graph $\Gamma$ such that $\mu(B_0 \cap F \cap \Gamma)\geq c_0\mu(B_0)$.
  \end{mainlemma}

  \subsection{Proof of Theorem \ref{mainthm}}  We first use Lemma \ref{mainlemma} to give the

 \begin{proof}[Proof of Theorem \ref{mainthm}] Suppose that $\mu$ is a non-atomic measure satisfying the assumptions of Theorem \ref{mainthm}.  As we discussed in the introduction (see \cite{Dav2}), since the Huovinen transform is bounded in $L^2(\mu)$, it follows that
 $$\sup_{x\in \C, r>0}\delta_{\mu}(B(x,r))<\infty,
 $$
 and, therefore, insofar as $\mathcal{H}^1(\supp(\mu))<\infty$, 
 \begin{equation}\label{posupdense}\Theta^*_{\mu}(x):=\limsup_{r\to 0}\delta_{\mu}(x,r)\in (0,\infty) \text{ for }\mu\text{-a.e. }x\in \C.
 \end{equation}
From Theorem \ref{JM2thm} we have that the Huovinen transform exists in principal value (i.e. (\ref{huovlimit}) exists).  Appealing to Lemma \ref{smoothnochange} we therefore infer that
\begin{equation}\label{CZprops}\|\wh{T}\|_{L^2(\mu), L^2(\mu)}<\infty \text{ and }\!\!\lim_{r_1, r_2\to 0}\wh{T}_{r_1,r_2}(\mu)(x)=0 \text{ for }\mu\text{-a.e. }x\in \C.
\end{equation}

Take an arbitrary subset $\wt{E}\subset \supp(\mu)$ with $\mu(\wt{E})>0$.   Our goal is to show that there is a Lipschitz curve that intersects $\wt{E}$ in a set of positive $\mu$-measure.  It is well known that this implies rectifiability -- For instance, by implying that the purely unrectifiable component of $\supp(\mu)$ has zero length,  see e.g. L\'eger \cite[p. 836]{L}. \\

Firstly, for each $i\in \mathbb{Z}$,  define 
\begin{equation} \label{split1}
    E_i=\{ x \in \wt{E}: 2^{-(i+1)}\leq \Theta^\ast_\mu(x) < 2^{-i} \}.
\end{equation}
The property (\ref{posupdense}) ensures that $\mu(\wt{E}\setminus \bigcup_i E_i)=0$. 

With a density threshold established, we now introduce $0 < \varepsilon_i \ll 1$, and put
\begin{equation*} 
    E_{i,j}=\{x \in E_i: \sup_{0 < r_1\leq r_2 < 1/j} |\wh{T}_{r_1,r_2}\mu(x)| \leq \frac{1}{C_k}\varepsilon_i 2^{-i-3}\},
\end{equation*}
and
\begin{equation} \label{split2}
    E_{i,j,m}=\{x \in E_{i,j}: \sup_{0 < r < 1/m} \kalpha_\mu(B(x,r)) \leq \frac{1}{C_k} \varepsilon_i^2 2^{-i-3} \},
\end{equation}
for $(j,m) \in \mathbb{N}^2.$  Here $C_k>1$ is the constant appearing in Lemma \ref{split}.

The assumption that $\lim_{r\to 0}\kalpha_{\mu}(B(x,r))=0$ for $\mu$-a.e. $x\in \C$, together with (\ref{CZprops}), imply that for every $i\in \mathbb{Z}$ $$\mu\Bigl(E_i\setminus \bigcup_{(j,m)\in \mathbb{N}^2} E_{i,j,m}\Bigl)=0.$$

Next we show that if $x \in E_{i,j,m}$, then 

\begin{equation} \label{split3}
   \frac{2^{-(i+2)}}{C_k}  \leq \limsup_{r \to 0} \wt{\delta}_\mu(B(x,r)) \leq  2^{-i+1}C_k.
 \end{equation}

Indeed, let $ r  \in (0, \frac{1}{30m})$ be such that $2^{-(i+2)}<  \delta_\mu(B(x,r)) < 2^{-i+1}$. Then $\kalpha_{\mu}(B(x, 30r))<\eps_i^2\delta_{\mu}(B(x,r))$, and Lemma \ref{split} is applicable ($\eps_i\ll 1$).  Consequently, $S_{x,r}(\varepsilon_i)\neq \varnothing$ and
$$ \frac{2^{-(i+2)}}{C_k} \leq \wt{\delta}_\mu(B(x,r)) \leq 2^{-i+1}C_k,$$
so the lower bound in (\ref{split3}) follows.  For the upper bound, recall that $C_k\geq 2/\lambda_k$ and so we infer from (\ref{trivupperbound}) that $\limsup_{r\to 0}\wt{\delta}_{\mu}(B(x,r))\leq C_k \Theta^{\ast}_{\mu}(x)<C_k 2^{-i}$. 

Next, we introduce $\theta_i\in (0,1)$ with $\theta_i \ll1$. Given $n \in \mathbb{Z}_+$, we define the sets $E_{i,j,m,n}$ as
\begin{align*}
    E_{i,j,m,n}=\Bigl\{x \in E_{i,j,m}:\frac{ 2^{-(i+2)}}{C_k} (1+\theta_i^2)^n  \leq & \limsup_{r \to 0} \wt{\delta}_\mu(B(x,r)) \\  < &   \frac{2^{-(i+2)}}{C_k}(1+\theta_i^2)^{n+1}\Bigl\}.
\end{align*}
Fixing a sufficiently large integer $N$ (depending on $C_k$ and $\theta_i$), we obtain from (\ref{split3}) the following decomposition:
 $$E_{i,j,m}= \bigcup_{n=0}^{N} E_{i,j,m,n}.$$

 Our final step is to further decompose $E_{i,j,m,n}$.  For $p\in \mathbb{N}$, set
$$E_{i,j,m,n,p}=\left\{x\in E_{i,j,m,n}: \sup_{0 < r \leq 1/p} \wt{\delta}_\mu(B(x,r)) \leq \frac{2^{-(i+2)}}{C_k}(1+\theta_i^2)^{n+2}\right\}.$$
Clearly, 
$$ E_{i,j,m,n}=\bigcup_{p\in \mathbb{N}} E_{i,j,m,n,p}.$$

Select  $\wt{E}_{i,j,m,n,p} \subset E_{i,j,m,n,p}$ satisfying $\wt{E}_{i,j,m,n,p}\cap \wt{E}_{i',j',m',n',p'}=\emptyset$ whenever $(i,j,m,n,p)\neq (i',j',m',n',p')$ but still 
$$\mu\left(\wt{E}\setminus \bigcup_{i,j,m,n,p}\wt{E}_{i,j,m,n,p}\right)=0.$$

Now fix $i,j,m,n,p$ with $\mu(\wt{E}_{i,j,m,n,p})>0$.  For each density point $z$ of $\wt{E}_{i,j,m,n,p}$ choose $r < \frac{1}{90}\min(1/j,1/k,1/p,1/30m)$ satisfying \begin{equation} \label{DensBo}
 \frac{ 2^{-(i+2)}}{C_k}(1+\theta_i^2)^{n-1}\leq \wt{\delta}_\mu(B(z,r)) \leq \frac{ 2^{-(i+2)}}{C_k}(1+\theta_i^2)^{n+2},   
\end{equation}
and 
\begin{equation} \label{rest}
    \mu(B(z,10r)\setminus \wt{E}_{i,j,m,n,p})< \frac{1}{ \lambda_k C_k}\eps_i \mu(B(z,r)).
\end{equation}

Consider the measure $\wt{\mu}:=\frac{1}{\delta_\mu(B_0)}\mu$,  where $B_0$ is a ball in $S_{z,r}(\varepsilon_i)$ (recall that $S_{z,r}(\varepsilon_i)\neq \varnothing$).  Our goal will be to apply Main Lemma \ref{mainlemma} to the measure $\wt{\mu}$ with ball $B_0$ and with $F$ taken to be a compact subset of $10B_0\cap \wt{E}_{i,j,m,n,p}$ with $\mu(10B_0\cap \wt{E}_{i,j,m,n,p}\backslash F)$ arbitrarily small.  Let us verify each of the assumptions of the lemma in turn:

 (a).  By definition $\delta_{\wt{\mu}}(B_0)=1$. Since $B_0 \in S_{z,r}(\varepsilon_i)$ it follows that $ \alpha_\mu(30B_0)\leq \varepsilon_i \delta_\mu(B_0)$ and therefore $\alpha_{\wt{\mu}}(30B_0)\leq \varepsilon_i.$
   
    Next we proceed to check that $\wt{\mu}(10B_0\setminus F) < \eps_i r_0 $.
     Provided $\mu(10B_0\cap \wt{E}_{i,j,m,n,p}\backslash F)$ is small enough,  (\ref{rest}) and the definition of $S_{z,r}(\varepsilon_i)$ ensure that
     $$ \mu(10B_0\setminus F) <  \frac{1}{C_k\lambda_k} \eps_i \mu(B(z,r)) \leq \eps_i \mu(B_0), $$
     which is the same as $\wt{\mu}(10B_0 \setminus F) \leq  \eps_ir_0.$

(b).  Fix $x \in F$ and $0 < r' \leq 90 r_0$.  We need to show that
    \begin{equation} \label{Growth}
     \wt{\delta}_{\wt{\mu}}(B(x,r'))\leq 1 + \theta_i.
    \end{equation}
    But since $r'<1/p$
    \begin{align*}
        \wt{\delta}_\mu(B(x,r')) & \leq \frac{2^{-(i+2)}}{C_k}(1+\theta_i^2)^{n+2} \leq (1+\theta_i^2)^3 \wt{\delta}_\mu(B(z,r))\\
        & \leq (1+\theta_i^2)^3 \delta_\mu(B_0)\leq (1+\theta_i)\delta_\mu(B_0),
    \end{align*}
   where  $(\ref{DensBo})$ was used in the second inequality, the third inequality follows from definition of $\wt{\delta}_{\mu}$, and the final inequality uses that $\theta_i\ll1$.
   The inequality $(\ref{Growth})$ is proved.
  
 The assumptions (c) and (e) hold since for all $x\in F$,
$$  \sup_{0 < r_1\leq r_2 < 1/j} |\wh{T}_{r_1,r_2}\mu(x)|\leq \frac{2^{-i-3}}{C_k} \varepsilon_i \text{ and }\!\!\!
\sup_{0 < r < 1/m} \kalpha_\mu(B(x,r)) \leq \frac{2^{-i-3}}{C_k} \varepsilon_i^2,$$
while $\delta_{\mu}(B_0) \geq \frac{1}{C_k}2^{-i-2}.$

(d)  Finally, since $\delta_\mu(B_0)\geq \frac{2^{-i-2}}{C_k}$, we have that $\|\wt T_{\wt{\mu},r}f\|_{L^2(\wt{\mu}), L^2(\wt{\mu})} \leq 2^{i+2} C_k\|T_{\mu}\|_{L^2(\mu)\to L^2(\mu)}$ for every $f \in L^2(\mu)$, so assumption (d) holds with $M$ replaced by $M_i=  2^{i+2} C_k \|T_{\mu}\|_{L^2(\mu)\to L^2(\mu)}$.
        
Therefore we have checked that the assumption of Main Lemma \ref{mainlemma} hold with $\eps_i$, $\theta_i$ and $M_i$.  Provided that $\eps_i$ and $\theta_i$ are sufficiently small in terms of $\max\{1,M_i\}$, with $ \eps_i$ much smaller than $ \theta_i$, we infer that there is a Lipschitz graph that intersects $\wt{E}$ in a set of positive measure.   \end{proof}

\subsection{Proof of Theorem \ref{PVthm}}

In this section, we indicate how Theorem \ref{PVthm} follows from Theorem \ref{mainthm} by using Corollary \ref{NTVcor}.

\begin{proof}[Proof of Theorem \ref{PVthm}]

Since $\mu$ is a finite measure satisfying (\ref{limsupmeas}), then $\supp(\mu)$ has $\sigma$-finite length.  We therefore infer from Corollary \ref{NTVcor} we may write
$\supp(\mu) = F\cup \bigcup_j E_j$, where $\mathcal{H}^1(F)=0$, $\mathcal{H}^1(E_j)<\infty$, and, with $\mu_j = \mu|_{E_j}$, the Huovinen transform is bounded in $L^2(\mu_j)$.  On the other hand, Theorem \ref{huovSLA} ensures that $\lim_{r\to 0}\kalpha_{\mu}(B(x,r))=0$ for $\mu$-almost everywhere.  From this it is a routine matter to see that $\lim_{r\to 0}\kalpha_{\mu_j}(B(x,r))=0$ for $\mu_j$-almost every density point $x$ of $\mu_j$.  But now we may apply Theorem \ref{mainthm} with the measure $\mu_j$.  Therefore $\mu$, as a countable union of rectifiable measures, is rectifiable.
\end{proof}

\section{Construction of the Lipschitz graph for the proof of the Main Lemma}\label{approxgraph}

Fix positive quantities $\delta, \varepsilon, \theta$ and $\alpha$ that will be determined later, satisfying $ \log \varepsilon \ll \log \theta \ll  \log \alpha \ll \log \delta \ll -1$.

Throughout this section we will assume that $\mu$ satisfies assumptions (a), (b) and (c) of  Main Lemma \ref{mainlemma} with these choices of $\eps$ and $\theta$.  The roles of $\delta$ and $\alpha$ will be introduced momentarily.

    We will adapt a version of the construction developed by L\'eger in \cite{L} (adapting work by David-Semmes \cite{DS} to the non-homogeneous setting) involving a stopping time construction.  The most significant distinction between the assumptions we have made in Main Lemma \ref{mainlemma} and those in \cite{L} is that we do not know know that the measure $\mu$ is flat (meaning that, say, $\alpha_{\mu}(B(x,r))$ is small at every $x\in F$ and $r<r_0$), but rather we only know that the measure $\mu$ is spike-flat ($\kalpha_{\mu}(B(x,r))$ is small if $x\in F$ and $r<r_0$).  Our main observation is that, due to the initial flatness assumption on $B_0$ (assumption (a) of Lemma \ref{mainlemma}) \emph{within the stopping time region} the measure must not only be spike-flat but truly flat (this is the content of Lemma \ref{CorStrip}), and so one can build an approximate Lipschitz graph (Proposition \ref{LipGraphProp}) as in the David-Semmes-L\'eger scheme.

Without loss of generality, we put $x_0=0$, $r_0=1$.

\begin{lemma}\label{startdens}  For every $x\in F$ and $r\in (0, 30)$,
$$\delta_{\mu}(B(x,r))\lesssim 1.
$$
\end{lemma}

\begin{proof} The statement is clear if $\delta_{\mu}(B(x,r))\leq 1$, so we may assume otherwise.  By assumption (c), $\kalpha_{\mu}(B(x,r))\leq \eps^2\leq \eps^2\delta_{\mu}(B(x,r))$, and the result follows from Lemma \ref{split} due to assumption (b).  \end{proof}

\subsection{The stopping time region}

 We set $B_0=B(0,1)$ and $D_0$ to be a line such that $$\alpha_{\mu,\mathcal{H}^1_{|D_0}}(\J B_0) \leq  2 \varepsilon.$$
 Without loss of generality, we may (and will) assume that $D_0=\R\times\{0\}$.

\begin{remark} \label{30dens}
An application of Lemma \ref{densitycomparison} tells us that since $\delta_\mu(B_0)=1$, we have that $$\delta_\mu(30B_0) \approx 1.$$
\end{remark}

\begin{definition}\label{Stotal} We define the region $S_{\text{total}}$ as the collection of pairs $(x,t) \in F \cap \overline{B_0} \times (0,20)$  satisfying the following two properties 
$$\begin{array}{ll}
(1)&\;  \delta_\mu(B(x,t)), \geq \delta \text{ and } \\
(2)& \text{ there exists } D \in \G_x \text{ with } \alpha_{\mu,D}(B(x,t)) \leq  \varepsilon\\& \text{ and }\angle (D, D_0) \leq \alpha. \end{array}$$
\end{definition}

  \begin{lemma} \label{start}
  There is a constant $C>0$ such that $$(F \cap \overline{B_0}) \times [C\sqrt{\eps}/\alpha, 12] \subset S_{\operatorname{total}}.$$
  \end{lemma}
  
\begin{proof}Fix $z\in F\cap \overline{B_0}$. From assumption (c), $\kalpha_{\mu}(B(z, 20))<\eps^2$, while trivially, $\delta_{\mu}(B(z, 20))\gtrsim \delta_{\mu}(B_0)\gtrsim 1$.  Consequently, part (2) of Lemma \ref{densitycomparison} yields that there is a constant $C>0$ such that $\delta_{\mu}(B(z,s))\gtrsim 1$ whenever $s\in [C\sqrt{\eps}, 12]$.

Now further assume that $s\in (q\frac{\sqrt{\eps}}{\alpha}, 12]$ for some $q>1$ to be determined momentarily. Since $\kalpha_{\mu}(B(z, 4s))\lesssim \eps^2\delta_{\mu}(B(z,s))$ and $\alpha_{\mu, D_0}(\J B_0)\lesssim \eps\delta_{\mu}(\J B_0)$, Lemma \ref{spikeflatter}\footnote{Applied with the role of $\delta$ played by a constant $\gtrsim 1$, and $B(x,r)=30B_0$ so that $B(z, 4s)\subset 2B(x,r)$.} yields that for some $D' \in \G_x$
$$\alpha_{\mu, D'}(B(z,s))\lesssim \eps^2\delta_{\mu}(B(z,s))<\eps\delta_{\mu}(B(z,s))
$$
where $\angle(D_0, D')\lesssim \sqrt{\eps}\frac{1}{s}\lesssim \alpha/q<\alpha$, provided that $q$ is chosen appropriately.
\end{proof}

\begin{definition}
 For $x \in  F \cap \overline{B_0},$ we set
    $$h(x)= \sup \{ t\in (0, 12]: (x,t) \not\in S_{\operatorname{total}}\}; $$and 
    $$S=\{ (x,t) \in S_{\operatorname{total}}: t \geq h(x) \}.$$
\end{definition}
 Notice that if $(x,t) \in S$, then $(x,t') \in S$ for $t' > t$, making $S$ a stopping time region.
 
On occasion we will abuse notation and write, for a ball $B$,  $B \in S$  (respectively $B \in S_{total}$) instead of $(c(B),r(B))\in S$ (respectively $(c(B),r(B))\in S_{total}$). 

We record a restatement of Lemma \ref{start} that will be used later on.
\begin{remark}\label{hsmall}
For $x \in F \cap \overline{B_0}$, $h(x) \lesssim \sqrt{\eps}/\alpha.$
\end{remark}

\subsection{Properties of the Stopping Time Region}

It will be convenient to set $$\bdary = \frac{\sqrt{\eps}}{\delta}.$$

   \begin{lemma} \label{CorStrip}
 Let $(x,r) \in S$ and $p \in \pi(B(x,r))$. Let $D \in \G_x$  satisfy that $\angle(D,D_0) \leq \alpha$ and $\alpha_{\mu,\mathcal{H}^1_{|D}}(B(x,r))\leq \varepsilon$.
 Then we have that 
 $$F \cap \pi^{-1}(B(p,r)) \subset B(x,3r)\cap\Bigl\{y \in \mathbb{C}: d(y,D) \lesssim  \bdary \cdot r\Bigl\}.$$
 \end{lemma}
 
 \begin{proof}
 Fix $z\in \pi^{-1}(B(p,r))\cap F$ and set $\wt{r}= \max(r,|x-z|)$.   Since $(x, \wt{r})\in S$, we have that $\delta_{\mu}(B(x,\wt{r}))\geq \delta$ and there exists $D' \in \G_x$ such that $$\alpha_{\mu, D'}(B(x,\wt r))\leq \eps\lesssim \frac{\eps}{\delta}\delta_{\mu}(B(x,\wt{r})) \text{ and } \angle(D',D_0)\leq \alpha.$$  But then $\delta_{\mu}(B(z, 2\wt{r}))\gtrsim \delta$ from which part 2 of Lemma \ref{densitycomparison} ensures that $\delta_{\mu}(B(z, \wt{r}))\gtrsim \delta$ (note here that, as $z\in F$, $\kalpha_{\mu}(B(z,2\wt{r}))\lesssim \frac{\eps^2}{\delta}\delta_{\mu}(B(z,2\wt{r}))$).  From here, Lemma \ref{startdens} ensures that $\delta_{\mu}(B(z, \wt{r}))\gtrsim \delta \cdot \delta_{\mu}(B(x, \wt{r}))$, and, since $B(z,\wt{r})\subset B(x,2\wt{r})$ we may apply Lemma \ref{linesdontmove}\footnote{with $\gamma$ replaced by $C\eps/\delta$, $\sigma = \mathcal{H}^1_{|{D'}}$, and $\nu$ a spike measure such that $\alpha_{\mu, \nu}(B(z, \wt{r}))<\eps^2\lesssim \frac{\eps}{\delta}\delta_{\mu}(B(z, \wt{r}))$} to conclude that
 \begin{equation}\label{tildedist}\dist(z, D')\lesssim  \frac{\sqrt{\eps}}{\delta} \wt{r}.
 \end{equation}
We next claim that $\wt{r}\leq 3r$.  If $|x-z|>3r$, then since the line $D' \in \G_x$  satisfies $\angle(D', D_0)\leq \alpha$, and $\dist(\pi(x), \pi(z))<2r$, it follows that
$$\wt{r} = |x-z|\lesssim \dist(z,D'),
$$
but given (\ref{tildedist}) this is absurd, and so $\wt{r}\in [r, 3r]$.  In particular, we have proved that $\pi^{-1}(B(p,r))\subset B(x, 3r)$.  Finally, Corollary \ref{linesdontmovecor} ensures that if we consider instead of $D'$ the line $D$ (which satisfies $\alpha_{\mu, D}(B(x,r))\leq \eps\lesssim \frac{\eps}{\delta}\delta_{\mu}(B(x,r))$), then $\angle(D, D')\lesssim \frac{\sqrt{\eps}}{\delta}$, and the result follows.\end{proof}
  
\begin{lemma}\label{notmuchturn}  Suppose $B, B'\in S$, $L>1$, $LB\cap LB'\neq \varnothing$, and $r(B')\leq r(B)$. Let $D_B$ and $D_{B'}$ be lines in $\G_{c(B)}$ and $\G_{c(B')}$ respectively satisfying that $\alpha_{\mu,D_B}(B) \leq \varepsilon$ and $\alpha_{\mu,D_{B'}}(B')\leq \varepsilon$. Then for all $y\in LB'\cap D_{B'}$,
$$\dist(y, D_B)\lesssim L^2\bdary \cdot r(B).
$$
\end{lemma}  
  
We will require the following simple result.

\begin{lemma}\label{samecenter}  Fix $\Lambda>1$.  Suppose that $B, \Lambda B\in S$.  Let $D_B$ and $D_{\Lambda B}$ be lines in  $\G_{c(B)}$ satisfying that $\alpha_{\mu,D}(B) \leq \varepsilon$ and $\alpha_{\mu,\Lambda D}(\Lambda B)\leq \varepsilon$, respectively. Then
$\dist(y, D_{\Lambda B})\lesssim \bdary\Lambda r(B)$ for every $y \in B\cap D_B$.
\end{lemma}

\begin{proof} Due to Lemma \ref{startdens}, $\delta_{\mu}(B)\gtrsim \delta\cdot\delta_{\mu}(\Lambda B)$, so application of Corollary \ref{linesdontmovecor} (with $B$ playing the role of $B(z,s)$ and $\Lambda B$ playing the role of $B(x,r)$) readily yields that
$$\min\{r(B), \dist(y, D_{\Lambda B})\}\lesssim  \bdary \cdot \Lambda r(B) \text{ for every } y \in D_B \cap B.
$$
But since $c(B)$ lies on $D_{\Lambda B}$, we obtain $\dist(y, D_{\Lambda B})\lesssim r(B)$ for $y\in B$, and the lemma is proved.
\end{proof}

\begin{proof}[Proof of Lemma \ref{notmuchturn}]
If $3LB\notin S$ then $Lr(B)\geq 3$ and we can replace $L$ by $L'\leq L$ where $3L'B\in S$ and $L'B\cap L'B'\neq \varnothing$.  (Recall that $B, B'$ have their centres on $B_0$.)  We therefore assume that $3LB\in S$.

Now, fix $\Lambda \approx L\frac{r(B)}{r(B')}$ such that both $3LB\supset \Lambda B'$ and $\Lambda B'$ belongs to $S$ (observe here that $r(3LB)\approx r(\Lambda B')$).    We first apply Lemma \ref{samecenter} twice to conclude that
\begin{equation}\label{2LBclose}\dist(y, D_{3LB})\lesssim L \bdary r(B) \text{ for all }y\in D_B\cap B
\end{equation}
and
\begin{equation}\label{lambdaprimeclose}\dist(y, D_{\Lambda B'})\lesssim L\bdary r(B) \text{ for all }y\in D_{B'}\cap B'.
\end{equation}
But now, since both $3LB$ and $\Lambda B'$ belong to $S$, have comparable radii, and $3LB\supset \Lambda B'$, we may use Corollary \ref{linesdontmovecor}\footnote{Here we appeal to Lemma \ref{startdens}, which ensures that $\delta_{\mu}(\Lambda B')\gtrsim \delta \delta_{\mu}(3LB)$}, from where
$$\dist(y, D_{3LB})\lesssim  L \bdary r(B) \text{ for all }y\in D_{\Lambda B'}\cap \Lambda B'.
$$
In combination with (\ref{lambdaprimeclose}), the previous inequality ensures that for every $y\in D_{B'}\cap B'$, there exists $z\in D_{3LB}\cap 3LB$ such that $d(y,z)\lesssim  L\bdary r(B)$.  Recalling that $D_{B'}$, $D_{\Lambda B'}$ and $D_{3LB}$ are lines, it follows that for every $y\in D_{B'}\cap LB'$, there exists $z\in D_{3LB}\cap 4LB$ such that $d(y,z)\lesssim  L^2\bdary r(B)$.  Now we infer from (\ref{2LBclose}) that there exists $w\in D_{B}$ with $d(z,w)\lesssim  L^2\bdary r(B),$ and the result follows.
\end{proof}
  
Although $B_0$ is not necessarily in $S$ ($0$ may not be in $F$), we still have the following results

\begin{cor}\label{stopclosetoB0} Suppose that $L>1$ and $B\in S$, and let $D_B \in \G_{c(B)}$ satisfying $\alpha_{\mu,D_B}(B)\leq \varepsilon$.  Then
$$\dist(y, D_0)\lesssim L^2 \bdary \text{ for every }y\in LB\cap D_B
$$
and therefore
$$\angle (D_B, D_0)\lesssim \frac{L^2}{r(B)}\bdary.
$$
\end{cor}
  
\begin{proof}  By part (2) Lemma \ref{densitycomparison}, $\mu(B(0,C \sqrt{\eps}))\gtrsim \sqrt{\eps}$ so $F\cap B(0, C\sqrt{\eps})\neq \varnothing$ (see assumption (a) of Main Lemma \ref{mainlemma}). Pick $D_{B_1} \in \G_x$ satisfying $\alpha_{\mu,D_{B_1}}(B_1)\leq \varepsilon$. Therefore we can choose a ball $B_1=B(x,10)\in S$ with $x\in F$ and $|x|\lesssim \sqrt{\eps}$.  We infer from Lemma \ref{notmuchturn} that
$$\dist(y, D_{B_1})\lesssim L^2 \bdary \text{ for all }y\in LB\cap D_B.
$$
But then it follows from Corollary \ref{linesdontmovecor} that $\dist(y, D_{B_1})\lesssim \sqrt{\eps}$ for all $y\in B_0\cap D_0$, and the Corollary follows. 
\end{proof}  
  
 The following Corollary follows from Lemma \ref{CorStrip} in an analogous manner to how the previous result follows from Lemma \ref{notmuchturn} (i.e. by finding a point $x\in F$ within a distance $\lesssim \sqrt{\eps}$ from $0$).  Moreover let us recall that $I_0=(-1,1)$.
  
\begin{cor}\label{CorStripUnit}  One has
$$F\subset \Bigl\{\dist(\cdot, D_0)\lesssim \bdary \Bigl\}.
$$
\end{cor}  
  
\subsection{Partition of the stopping scales}  
  We define the following three disjoint subsets of $F\cap \overline{B_0}$:
\begin{align*}\nonumber
&\Z=\{ x \in F\cap \overline{B_0}: h(x) = 0 \},\\
& F_1=\{ x \in F \cap\overline{B_0} \setminus \Z: \delta_{\mu}(B(x, h(x)))\leq \delta\}, \text{ and }\\
&F_2= \Bigl\{ x \in F\cap \overline{B_0} \setminus (\Z\cup F_1): \!\!\!\begin{array}{l}\text{ there is } D \in \G_x \text{ with }\angle(D,D_0) \geq \alpha\\\text{ and } \alpha_{\mu,D}(B(x,h(x))) \leq \eps\end{array}\!\!\!\Bigl\}.\end{align*}

    \begin{lemma} One has    $$F=\Z \cup F_1 \cup F_2. $$
    \end{lemma}
    
    \begin{proof} Fix $x\in F\backslash(\Z\cup F_1)$. Therefore $h(x)>0$ and $\delta_{\mu}(B(x, h(x)))> \delta$. Moreover, $(x, 4h(x))\in S_{\text{total}}$ so there exists $D \in \G_x$ with $$\alpha_{\mu,D}(B(x, 4h(x)))\leq\eps\lesssim \frac{\eps}{\delta}\delta_{\mu}(B(x, 4h(x))).$$ Since $\kalpha_{\mu}(B(x,4h(x)))\lesssim \frac{\eps^2}{\delta}\delta_{\mu}(B(x,4h(x)))$ and $\delta_{\mu}(B(x, h(x)))\gtrsim \delta_{\mu}(B(x, 4h(x)))$ (the latter inequality holding, for instance, by part (1) of Lemma \ref{densitycomparison}), we have from Lemma \ref{spikeflatter} and Lemma \ref{startdens} that there exists $D' \in \G_x$ such that
    \begin{equation}\label{flatterthanuthink}\alpha_{\mu,D'}(B(x,h(x)))\lesssim \frac{\eps^2}{\delta^2}\delta_\mu(B(x,h(x)))\lesssim \frac{\eps^2}{\delta^2}.
    \end{equation}
      Notice that if $(x,h(x)) \notin S_{total}$, then by the definition of $S_{total}$ (Definition \ref{Stotal}) we have that $\angle(D',D_0)>\alpha$ and therefore $x \in F_2$.
    Consequently, we may assume that $(x,h(x))\in S_{total}$.     By the definition of $h(x)$ there exists $r_j \to h(x)$ with $r_j< h(x)$ such that the balls $B(x,r_j)$ fail to satisfy one of the properties (1) or (2) in the definition of $S_{\text{total}}$. 
    
   But if a countable number of the balls $B(x, r_j)$ were to satisfy that $\delta_{\mu}(B(x, r_j))<\delta$ then $\delta_{\mu}(B(x,h(x)))\leq\delta$, which is not the case.  Similarly, if $\alpha_{\mu}(B(x, r_j))\geq \eps$ for infinitely many $j$ then by continuity of the alpha numbers, see Lemma \ref{continuity}, we have $\alpha_{\mu}(B(x, h(x)))\geq \eps$,  contradicting (\ref{flatterthanuthink}).  Therefore there exist lines $D_j$ through $x$ and radii $r_j \to h(x)$ with $r_j < h(x)$, $\angle (D_j, D_0)> \alpha$, and $\alpha_{\mu, D_j}(B(x, r_j))\leq \eps$.  We may pass to a subsequence if necessary to obtain that $D_j$ converge (locally) to a line $\wt{D}$ with $\angle (\wt{D}, D_0)\geq \alpha$.
    But then the continuity of the transportation coefficients (Lemma \ref{continuity}) ensures that  $\alpha_{\mu, \wt{D}}(B(x, h(x)))\leq \eps$, and hence $x\in F_2$.
    \end{proof}
  
\begin{remark}\label{F2alpharem}  An application of Corollary  \ref{linesdontmovecor} ensures that if $x\in F_2$ and $D'$ is any line in $\G_x$ for which $\alpha_{\mu, D'}(B(x,h(x)))\leq \eps$, then $\angle(D', D_0)\geq \alpha-C\tau \geq \frac{\alpha}{2}$.
\end{remark}

We shall show momentarily that $\Z$ lies in the zero set of a Lipschitz continuous function.  We will therefore want to show that the measure of the sets $F_1$ and $F_2$ is small.

 \subsection{Regularization of $h$}    The function $h$ itself can be quite irregular, so, as is standard, we proceed to introduce the functions $d$ and $D$.
 
 \begin{definition}
    For $x \in \C$, we set
$$d(x)= \inf_{(X,t)\in S} (|X-x| + t),$$
and for $p \in D_0$, 
$$D(p)=\inf_{x \in \pi^{-1}(p)} d(x)= \inf_{(X,t)\in S} (d(\pi(X),p)+t).$$
\end{definition}

    \begin{remark} Observe that
\begin{enumerate}
    \item the functions $d$ and $D$ are $1$-Lipschitz functions and
    \item $h(x)\geq d(x)$ for every $x\in F\cap\overline{B_0}$.
\end{enumerate}
\end{remark}
\begin{lemma}
We have that 
$$\Z = \{x \in \C: d(x)=0\}  = \{x \in F\cap B_0: d(x)=0\}.$$
\end{lemma}
  \begin{proof}
   If $x\notin \overline{B_0}\cap F$ then $d(x)>0$, so since $d\leq h$ on the closed set $F\cap \overline{B_0}$, we have $$\Z \subset \{x \in \C: d(x)=0\}  = \{x \in F\cap B_0: d(x)=0\}.$$ Next, we prove that if $x\in \C$ satisfies $d(x)=0$ then $h(x)=0$. If $d(x)=0$, then certainly $x \in F\cap \overline{B_0}$. Fix $\tau >0$. We can find a sequence of pairs $(x_j,\tau_j) \in S$ with $x_j \in F$, $x_j \to x$, and $\tau_j \to 0$ with $\tau_j < \tau$ for every $j$. In particular, $(x_j,\tau) \in S$ for every $j$. Since for any $\tau'\in (0, \tau)$, $\delta_{\mu}(B(x_j, \tau'))\geq \delta$ for sufficiently large $j$, it follows that that $\delta_\mu(B(x,\tau))\geq \delta.$ 
   
   Let $D_j \in \G_{x_j}$ be lines with $\alpha_{\mu,D_j}(B(x_j,\tau))\leq  \varepsilon$ and $\angle(D_j,D_0)\leq \alpha$. Appealing to Lemma \ref{continuity}, we obtain that (after passing to a subsequence if necessary)  there exists $D \in \G_x$  with $\angle(D,D_0)\leq \alpha$ such that $\alpha_{\mu,D}(B(x,\tau))\leq \varepsilon$. Since $\tau>0$ is arbitrary, the statement follows.
  \end{proof}

\subsection{The Lipschitz Mapping}  The next step is to construct a Lipschitz mapping with Lipschitz constant $\lesssim \alpha$ whose graph is close to points in $F$.  Recall that $I_0=(-1,1)$.

\begin{prop}\label{LipGraphProp}  There exists a Lipschitz continuous function $\A:\R\to \R$ satisfying $\supp(\A)\subset 3I_0$, $\|\A\|_{\Lip}\lesssim \alpha$, such that, with $\wt\A(p) = (p, \A(p))$ and $\Gamma = \{\A(p): p\in \R\}$,  the following properties hold:
\begin{enumerate}
\item $|\A''(p)| \lesssim \frac{\bdary}{D(p)}$ for any $p\in \R$,
\item $\Gamma \subset \Bigl\{\dist(\cdot, D_0)\lesssim \bdary\Bigl\},$
\item If $x\in F$, then
$$|\wt\A(\pi(x))-x|\lesssim \bdary \cdot D(\pi(x)).
$$
(In particular, $\Z\subset \Gamma$.)
\item If $B(x,r)\in S$ and $D\in \G_x$ satisfies $\alpha_{\mu,D}(B(x,r))\leq \eps$, then for every $p\in \pi(B(x,r))$, 
$$\dist(\wt\A(p), D)\lesssim \bdary \cdot r.
$$
\end{enumerate}
\end{prop}
  
Given the strong flatness property proved in Lemma \ref{CorStrip} (along with Lemma \ref{notmuchturn}, which informally states that good approximating lines for balls $B\in S$ do not change much locally), the reader familiar with the L\'eger scheme will likely find few obstacles in providing the proof of Proposition \ref{LipGraphProp} for themselves by modifying either \cite{L} or Chapter 7 of \cite{To5}.   However, since there are some minor changes required, we provide a relatively detailed treatment in Appendix \ref{LipConstructAp}.   
    
  \subsection{Density of $\mu$ under the projection to $D_0$}
  Our next lemma concerns the density of the projection of $\mu_{|F}$ to $D_0$.  This is a key property required to run the scheme of Tolsa which will show that the set $F_2$ has small measure.  Set $\sigma$ to be the Borel measure on $\R$ given by
  $$\sigma = \pi_{\#}(\mu_{|F}), \text{ so }\sigma(E) = \mu(F\cap \pi^{-1}(E))\text{ for a Borel set }E\subset \R.
  $$
  
  \begin{lemma}\label{sigmabd}
  One has
 \begin{equation} \label{ProjGrowth}
    \sigma(B(p,r))\leq (1+C\alpha^2)2r, \text{ for } p \in \R \text{  and } r \in ( \eps^{1/4}D(p),1).
 \end{equation}
 \end{lemma}
 
 \begin{proof}   
 Without loss of generality we may assume that $p\in 10I_0$ (recall that $F\subset 10B_0$).

\textbf{Case 1:} $r<\frac{\sqrt[4]{\eps}}{100}$.  

 Fix $t = r/\sqrt[4]{\eps}$, so $t>D(p)$ and there is $x \in \pi^{-1}(p)$ with $d(x)<t$.  Therefore we can find $(X,s)\in S$ with $|x-X|+s<t$, and so $B(X, 3t)\in S$.  Notice that $\pi(B(X, 3t))\supset B(p, t)$, and so appealing to Lemma \ref{CorStrip},
$$F\cap \pi^{-1}(B(p,t))\subset B(X, 6t)\cap \Bigl\{y\in \C: \dist(y, D)\lesssim \bdary t\Bigl\}
$$
for a line $D$ through $X$ with $\angle(D, D_0)\leq \alpha$.

Consequently, since $r =\eps^{1/4}t$, then $F \cap \pi^{-1}(B(p,r))$ is contained in a strip of width $C\frac{\bdary}{\eps^{1/4}}r\ll\sqrt{\lambda} r$ around a line $D$ with $\angle(D,D_0)\leq \alpha$. Therefore, if $z=\pi^{-1}(p)\cap D$, then $F\cap \pi^{-1}(B(p,r))\subset B(z,R)$ where $$r\leq R \leq 
\bigl(1 + \alpha^2 +C\sqrt{\bdary}\bigl)r\leq (1+C\alpha^2)r.$$  

Since $X\in F$, assumption (b) in the Main Lemma ensures that $\wt\delta_{\mu}(B(X, t))\leq 1+\theta$, and since $(X,t)\in S$, with $t<1/50$, we have that $\alpha_{\mu}(B(X, \J t))
\leq\eps$.  Since $B(z,R)\subset B(X, 30t)$, Lemma \ref{densG} is applicable with $x$ replaced by $X$, $r$ replaced by $t$, and $B' = B(z,R)$.  From the conclusion of this lemma it follows (recall that $\theta\ll \alpha^2$) that $$\delta_\mu(B(z,R))\leq1+\theta+C\eps^{1/8},$$ so $\mu(B(z,R))\leq (1+C\alpha^2)2r,$
and the required statement follows.\\

\textbf{Case 2:}  $r \geq \frac{\sqrt[4]{\eps}}{100}$.  In this case we apply the argument above with the role of the ball $B(x,t)$ replaced by $B(0,1)$.  We have from Corollary \ref{CorStripUnit} that $F\subset 10B_0\cap \{\dist(\cdot, D_0)\lesssim \sqrt{\lambda}D(p)\}$.  On the other hand, $\alpha_{\mu}(30B_0)\leq \eps$, and, although $0$ need not belong to $F$, the fact that $\delta_{\mu}(B_0)=1$ implies $\wt\delta_{\mu}(B_0)\leq 1$, which suffices to apply Lemma \ref{densG}.  (One can actually get a bound that only depends on $\bdary$ (and not $\alpha$) in this case, but we will not need this improvement.)\end{proof}
    
  \section{Size of $F_1$}\label{F1sec}
 
 The proof of the following result can be found as Proposition 3.19 in \cite{L} or Lemma 7.33 in \cite{To5}.
  \begin{prop}\label{F1small} One has
  $$\mu(F_1)\lesssim \delta\ll 1.$$
  \end{prop}
  
Every point $x\in F_1$ is the centre of a ball $B(x, h(x))$ which is of low density ($\leq \delta$), but $x$ is also lies very close to the Lipschitz graph $\Gamma$ (in the sense that $\dist(x, \Gamma)\lesssim \bdary d(x)\lesssim \bdary h(x)\ll h(x)$ for every $x\in F_1$).   From these observations the  Besicovitch covering lemma readily allows us to establish Proposition \ref{F1small}.\\
  
 \section{The size of $F_2$}\label{sizeF2}
 
Given Proposition \ref{F1small}, our goal is now to show that $\mu(F_2)$ is also small.  

Our goal will be to verify the following proposition.

\begin{prop}\label{F3small} Provided $\alpha \ll 1$ and $\log \eps \ll \log \alpha$,
$$\mu(F_2)\leq \sqrt{\alpha}.$$
\end{prop}

We start by recording the following estimate that can be found as Lemma 10.1 in \cite{To3} or Lemma 7.34 of \cite{To5}.  See also Section 5 of \cite{L}.  We give a self-contained proof.

Set $\|f\|^2_{L^2(\R)} = \int_{\R}|f|^2 dm_1$, where $m_1$ is the Lebesgue measure on $\R$.

\begin{lemma} \label{BigGrad} We have 
$$\mu(F_2) \lesssim \alpha^{-2}\|\A'\|_{L^2(\R)}^2.$$
\end{lemma}

\begin{proof}  
Suppose $x\in F_2$, so $\delta_{\mu}(B(x, h(x)))\geq \delta$.  Recall from Remark \ref{F2alpharem} we have that any $D\in \G_x$ for which
$$\alpha_{\mu, D}(B(x, h(x)))\leq \eps\text{ satisfies }  \angle(D, D_0)\geq \alpha/2. 
$$  Take a sequence of radii $r_{n}\to h(x)$, $r_n>h(x)$ such that the associated lines $D_n \in \G_x$ satisfying $\alpha_{\mu, D_n}(B(x, r_n))\leq \eps$ converge to a line $D$ such that $\alpha_{\mu, D}(B(x, h(x)))\leq \eps$ holds (and so $\angle(D,D_0)\geq \alpha/2$).

Pick $p\in \pi(B(x, h(x)))$.  We claim that
\begin{equation}\label{distancetolineA}\dist(\wt\A(p), D)\lesssim \bdary h(x)\ll\alpha \cdot h(x).
\end{equation}
To see this, note that $B(x, r_n)\in S$.  Then by property (4) of Proposition \ref{LipGraphProp},
$$\dist(\wt\A(p), D_{n})\lesssim \bdary r_n,
$$
letting $n\to \infty$ we obtain the claimed inequality.

Choose $p,q\in \pi(B(x,h(x)))$, with $|p-q|\gtrsim h(x)$.  Then since $\angle(D, D_0)\gtrsim \alpha$,
$$\alpha \cdot h(x) \stackrel{(\ref{distancetolineA})}{\lesssim} |\A(p)-\A(q)|\lesssim \int\limits_{I(\pi(x), h(x))}|\A'|\,dm_1,
$$
where the second inequality is a straightforward consequence of the fundamental theorem of calculus.  Using the Cauchy-Schwarz inequality and Lemma \ref{startdens}, we therefore obtain that
\begin{equation}\label{alphalowerbd}\alpha^2\cdot \mu(B(x, 30h(x)))\lesssim \alpha^2h(x)\lesssim \int_{I(\pi(x), h(x))}|\A'|^2\,dm_1.
\end{equation}

On the other hand, since $(x,2h(x))\in S$, it is immediate from Lemma \ref{CorStrip} that if $y\in F$ and $B(x, 6h(x))\cap B(y, 6h(y))=\varnothing$, then $$I(\pi(x), 2h(x))\cap I(\pi(y), 2h(y))=\varnothing.$$  From the Vitali covering lemma, we choose a subcollection of the balls $B(x, 6h(x))$, say $B(x_j, 6h(x_j))$, that are pairwise disjoint, and satisfy $\bigcup_jB(x_j, 30h(x_j))\supset F_2$.  But then the intervals $I(\pi(x_j), h(x_j))$ are pairwise disjoint, so by summing (\ref{alphalowerbd}) we obtain
$$\alpha^2\mu(F) \lesssim \int_{3I_0}|\A'|^2\,dm_1.
$$
The result is proved.
\end{proof}

\section{Calder\'on-Zygmund operators on Lipschitz graphs with small constant}\label{CZsmallconstant}

Like in Tolsa's work \cite{To4}, the behavior of Calder\'on-Zygmund operators on Lipschitz graphs with small Lipschitz constant plays an important role in our work.  Here we carry out a suitable adaptation to the Huovinen kernels.  The main point is that, on a Lipschitz graph with small constant, the normal component of the Huovinen kernel behaves like a small perturbation of the normal component of the Cauchy kernel.

Recall that $$K_k^{\perp}(z) = \frac{\Im (z^k)}{|z|^{k+1}}\text{ for }z\in \C.$$

Throughout this section, we will denote by $\A:\R\to \R$ a compactly supported Lipschitz continuous function with $\|\A'\|_{\infty}\leq 1$.  We set $\wt{\A}(t) = (t, \A(t))(= t+i\A(t)\in \C)$, and $\Gamma = \{\wt{\A}(t):t\in \R\}$.  

The goal of the section is to derive the following result:

\begin{thm}\label{Huovbds}  There exists constants $C, c>0$ and $\alpha_0>0$ depending on $k$ such that if $\|\A'\|_{\infty}\leq \alpha_0$, and $\diam(\supp(\A))\lesssim 1$, then
\begin{enumerate}
\item for every $p\in (1,\infty)$, the principal value operator associated $K_k^{\perp}$ has operator norm at most $C_p \|\A'\|_{\infty}$, and 
\item we have the lower bound
\begin{equation}\label{lwbdleb1}
\int_{\Gamma}\Bigl|P.V.\int_{\Gamma}K_k^{\perp}(z-\omega)d\mathcal{H}^1(\omega)\Bigl|^2d\mathcal{H}^1(z) \geq c\|\A'\|^2_{L^2(\R)}-C\|\A'\|_{\infty}^4.
\end{equation}
\end{enumerate}
\end{thm}

For $t\in \R$, we shall set
$$J(\wt{\A})(t) = \sqrt{1+\A'(t)^2},$$
so that for any $f\in L^1(\Gamma)$,
\begin{equation}\label{changevariable}\int_{\C}f(\omega)d\mathcal{H}^1_{|\Gamma}(\omega) = \int_{\R}f(\wt{\A}(t))J(\wt{\A})(t)\,dm_1(t),\end{equation}

Using (\ref{changevariable}), we shall prove bounds of the the operator norm in $L^2(\Gamma)$ of the Calder\'on-Zygmund operator
$$T^\perp(f\mathcal{H}^1_{|\Gamma})(z) = P.V.\int_{\C}K_k^{\perp}(z-\omega)f(\omega)d\mathcal{H}^1_{|\Gamma}(\omega)$$
by first considering the principal value operator norm in $L^2(\R)$ of the operator
$$T_{A}(g)(t) = P.V.\int_{\R}K_k^{\perp}(\wt \A(t)- \wt \A(s)) g(s)\,dm_1(s),\;\; t\in \R.
$$
The following theorem is a well known result regarding Calder\'on commutators, see \cite[Chapter 2]{Dav2} for an exposition including several approaches to how it can be proved.
\begin{thm}[Boundedness of Calder\'on commutators]\label{calderoncommutators} There exists $C_1>0$ such that for every $p\in (1,\infty)$ and $\ell\in \mathbb{N}$, the CZO acting on $L^p(\R)$ with kernel $$K(t,s) = \frac{1}{t-s}\Bigl(\frac{\A(t)-\A(s)}{t-s}\Bigl)^{\ell}$$
is a bounded principal value operator in $L^p(\R)$ with norm $\lesssim_pC_1^{\ell} \|\A'\|^{\ell}$.
\end{thm}

We next recall an important tool in our argument, which is a special case of \cite[Theorem 1.3]{To4}, relying ultimately on a Fourier analytic argument.  

\begin{thm}\label{tolsalowbd} There exists $\alpha_0>0$ such that if $\|A'\|_{\infty}\leq \alpha_0$, then
$$\int_{\R}\Bigl|P.V. \int_{\R}\frac{\A(t)-\A(s)}{(t-s)^2}\,dm_1(s)\Bigl|^2 \,dm_1(t) \gtrsim \|\A'\|^2_{L^2(\R)}.
$$ \end{thm}

We now examine the difference between normal components of the Huovinen and Cauchy transforms.   For $|s|<|t|$, we may expand the kernel
\begin{equation}\begin{split}\label{kernelexpand}K_k^{\perp}(t+is) &=  \frac{\Im[(t+is)^k]}{(t^2+s^2)^{(k+1)/2}}
 = \sum_{\ell\in \mathbb{N}, \ell\text{ odd }}c_{k,\ell}\frac{s^{\ell}}{t^{\ell+1}},\end{split}\end{equation}
where $c_{k,\ell}\in \R$ satisfy 
\begin{equation}\label{coefficientbds}c_{k,1}=k\text{ and }\sum_{\ell} |c_{k,\ell}|\Bigl(\frac{1}{2}\Bigl)^{\ell}\lesssim_k 1.
\end{equation}

Consequently,  we see that
\begin{equation}\label{Kkseries}K_k^{\perp}(\wt{\A}(t)-\wt{\A}(s)) = k\frac{\A(t)-\A(s)}{(t-s)^2}+\sum_{\ell\geq 3, \ell\text{ odd}}c_{k,\ell}\frac{(\A(t)-\A(s))^{\ell}}{(t-s)^{\ell+1}}.
\end{equation}

Now, if $\|\A'\|_{\infty}\leq \alpha_0$ for a small enough $\alpha_0$, the kernel 
$$K_{\text{tail}}(t,s) = \sum_{\ell\geq 3, \ell\text{ odd}}c_{k,\ell}\frac{(\A(t)-\A(s))^{\ell}}{(t-s)^{\ell+1}}
$$
is a Calder\'on-Zygmund kernel, and Theorem \ref{calderoncommutators} ensures that, for any $p\in (1,\infty)$, the associated principal value operator is bounded in $L^p(\R)$ with norm $\lesssim_{p,k}\|A'\|_{\infty}^3$.  Therefore
\begin{enumerate}
\item[(a)]  the principal value operator with kernel $K_k^{\perp}(\wt{\A}(t)-\wt{\A}(s))$ has $L^p(\R)$ operator norm $\lesssim_{k,p} \|\A'\|_{\infty}$, 
\item[(b)] employing a simple localization argument yields that
$$\int_{\R}\Bigl|\int_{\R}K_{\text{tail}}(t,s)\,dm_1(s)\Bigl|^2\,dm_1(t)\lesssim \|\A'\|_{\infty}^6\diam(\supp(\A)),
$$
\item[(c)] if $\|\A'\|_{\infty}$ is small enough and $\diam(\supp(\A))\lesssim 1$, then part (b) and Theorem \ref{tolsalowbd} ensures that there are constants $C,c$ depending on $k$ such that
\begin{equation}
\nonumber
\int_{\R}\Bigl|\int_{\R}K_k^{\perp}(\wt{\A}(t)-\wt{\A}(s))\,dm_1(s)\Bigl|^2\,dm_1(t) \geq c\|\A'\|^2_{L^2(\R)}-C\|\A'\|_{\infty}^6.
\end{equation}
\end{enumerate}

Finally, observe that $|J(\wt{\A})(t)-1|=|\sqrt{1+|\A'(t)|^2}-1|\lesssim |\A'(t)|^2$.  Consequently, Theorem \ref{Huovbds} now follows from the change of variable formula (\ref{changevariable}), employing the bound on the operator norm (a) to bound the errors accumulated from passing from $\R$ to $\Gamma$.

 \section{The main comparison estimates}

Recall that our main goal is to prove Proposition \ref{F3small}.  We therefore assume that $\mu$ satisfies the assumptions of Main Lemma \ref{mainlemma}, and introduce  $\delta, \varepsilon, \theta$ and $\alpha$ satisfying $ \log \varepsilon \ll \log \theta \ll  \log \alpha \ll \log \delta \ll -1$, so that the construction of Section \ref{approxgraph} is valid.
  
For $x \in \C$,  set
    $$\ell(x)= \frac{1}{10} D(\pi(x)).$$ 
   We recall that we set $$\bdary = \frac{\sqrt{\eps}}{\delta},$$ so that (see (b) and (c) of Proposition \ref{LipGraphProp})
  \begin{equation}\label{betabdary}\Gamma \subset \{x\in \C:\dist(x, D_0)\lesssim \bdary\} \text{ and }
   \end{equation}
   \begin{equation}\label{Fbdary}F\subset \{x\in \C: \dist(x, \wt{A}(\pi(x)))\lesssim \bdary \ell(x)\}.
   \end{equation}
  
    Denote for any measure $\nu$ 
    $$T^\perp_{\ell(\cdot),1} \nu (x)= \HT^\perp_{\ell(x)}\nu(x) - \HT^\perp_{1}\nu(x).$$
        
    Put $I_0 = (-1,1)\subset D_0$.
    
   The goal of this section will be to prove the following result:
   
   \begin{prop}\label{Tmulowbd}There is a constant $C>0$ such that, as long as $\alpha \ll 1$, and $\log \bdary \ll \log \alpha$, 
   $$\|T^{\perp}_{\ell(x), 1}(\mu)\|_{L^2(\mu_{|F\cap \pi^{-1}(4I_0)})}\gtrsim \|\A'\|_{L^2(\R)}-C\alpha^{2}.
   $$
   
   \end{prop}

 We shall set $\|f\|_{L^2(\Gamma)}^2 = \int_{\Gamma}|f|^2d\H^1$.
 
 Recall that, since $\|\A'\|_{\infty}\lesssim \alpha$, applying Theorem \ref{Huovbds} yields
\begin{equation} \label{gradL2}
    \|\A'\|_{L^2(\R)} - C \alpha^2  \lesssim \| P.V. \; T^\perp \mathcal{H}^1_{|\Gamma} \|_{L^2(\Gamma)}.
 \end{equation}
 provided that $\alpha\ll 1$.  Comparing this estimate with Proposition \ref{Tmulowbd}, our goal is to (essentially) replace $\H^1|_{\Gamma}$ with by $\mu_{|F}$ on the right hand side of (\ref{gradL2}).

\subsection{Localization estimates}

\begin{lemma}[Localization lemma]\label{GammatoStop} For every $p\in (1,\infty)$,
$$\Bigl|\|T^\perp(\mathcal{H}^1_{|\Gamma})\|_{L^p(\Gamma)} - \|T^\perp_{\ell(\cdot),1}(\mathcal{H}^1_{|\Gamma})\|_{L^p(\Gamma \cap \pi^{-1}(4I_0))}\Bigl|\lesssim_p \alpha^2.$$

\end{lemma}

\begin{proof}
We recall that $\supp(\A)\subset \pi(3B_0)$.   Observe that
\begin{align*}
\Bigl|\|T^\perp(\mathcal{H}^1_{|\Gamma})\|_{L^p(\Gamma)} &- \| T^\perp(\mathcal{H}^1_{|\Gamma})\|_{L^p(\Gamma \cap \pi^{-1}(4I_0))} \Bigl |\leq \|\chi_{\Gamma\setminus \pi^{-1}(4I_0)} T^\perp(\mathcal{H}^1_{|\Gamma)}\|_{L^p(\Gamma)}
\end{align*}
Take $x \in \Gamma\setminus \pi^{-1}(4I_0)=D_0\setminus 4I_0$ (so $\pi^\perp(x)=0$), and we set
\begin{align*}
    |T^\perp(\mathcal{H}^1_{|\Gamma})(x)|&\leq \int_{y \in \Gamma} \frac{\dist(y,D_0)}{|x-y|^2} \,d\mathcal{H}^1(y)= \int_{y \in \Gamma \cap \pi^{-1}(3I_0)} \!\!\!\!\frac{\dist(y,D_0)}{|x-y|^2} \,d\mathcal{H}^1(y) \\
    & \lesssim \frac{1}{(1+|x|)^2}\int_{y \in \Gamma\cap \pi^{-1}(3I_0)} \!\!\!\!\!\!\dist(y,D_0) \,d\mathcal{H}^1(y)  \stackrel{(\ref{betabdary})}{\lesssim}  \frac{\bdary}{(1+|x|)^2}.
\end{align*}
Raising this inequality to the power $p$ and integrating on $D_0 \setminus 4I_0$, we obtain
$$ \|\chi_{\Gamma\setminus \pi^{-1}(4I_0)} T^\perp(\mathcal{H}^1_{|\Gamma})\|_{L^p(\Gamma)}
\lesssim \bdary.$$

For $x=\wt{\A}(t)$ for $t\in 4I_0$, write
\begin{align*}| T^\perp(\mathcal{H}^1_{|\Gamma})(x) - T^\perp_{\ell(x),1}(\mathcal{H}^1_{|\Gamma})(x)|  &\leq |S(x)| + |\HT^\perp_{1}(\mathcal{H}^1_{|\Gamma})(x)|,
\end{align*}
with\footnote{The integral $S$ is a principal value integral, but we shall suppress the $\text{P.V.}$ notation in principal value integrals whenever it is clear from context (in order to save line space).}
$$S(x)= \int \Bigl(1 -\Psi\Bigl(\frac{|\wt{\A}(t)-\wt{\A}(s)|}{D(t)/10}\Bigl)\Bigl) \frac{\Im(\wt{\A}(t)-\wt{\A}(s))^k}{|\wt{\A}(t)-\wt{\A}(s)|^{k+1}} J(\wt{\A})(s) \,dm_1(s),$$
where $J(\wt{\A}) = \sqrt{1+|\A'|^2}$, $x= \wt{\A}(t)$, $y \in \wt{\A}(s)$, with $t,s \in \R$.

 The estimate for second term is straightforward:
\begin{align*} |\HT^\perp_{1}(\mathcal{H}^1_{|\Gamma})(x)|
    &\lesssim \int\limits_{\substack{y \in \Gamma:\\ |y-x|\geq 1/2}}  \frac{|\pi^\perp(x)-\pi^\perp(y)|} {|x-y|^2} \,d\mathcal{H}^1(y) \\
 ( \supp(A)\subset 3I_0)  & \lesssim \int\limits_{\substack{y \in \Gamma:\\|y-x|\geq 1/2}} \frac{|\pi^{\perp}(x)|}{|x-y|^{2}} \,dm_1(y)+ \!\!\!\!\int\limits_{\substack{y\in \pi^{-1}(3I_0)\cap \Gamma\\|x-y|>1/2}}\frac{|\pi^\perp(y)|} {|x-y|^2}d\mathcal{H}^1(y) \\&\lesssim \dist(x,D_0) +\int_{\Gamma\cap \pi^{-1}(3I_0)}\dist(y, D_0)d\mathcal{H}^1(y)\\(\ref{betabdary})&\lesssim \dist(x, D_0)+\bdary.
\end{align*}
Therefore, using (\ref{betabdary}) once again
$$\|\HT^\perp_{1}(\mathcal{H}^1_{|\Gamma})\|_{L^p(\pi^{-1}(4I_0)\cap \Gamma)} \lesssim \Bigl(\int_{\Gamma\cap 4\pi^{-1}(I_0)} \dist(x, D_0)^p d\mathcal{H}^1(x)\Bigl)^{1/p}+\bdary \lesssim \bdary .$$

The estimate of $S(x)$ will take more work. We split \begin{align*}
    S(x)& =\int \Bigl(1 -\Psi\Bigl(\frac{t-s}{D(t)/10}\Bigl)\Bigl) \frac{\Im((\wt{\A}(t)-\wt{\A}(s))^k)}{|\wt{\A}(t)-\wt{\A}(s)|^{k+1}} \,dm_1(s) \\
    & + \int \Bigl(\Psi\Bigl(\frac{t-s}{D(t)/10}\Bigl) -\Psi\Bigl(\frac{|\wt{\A}(t)-\wt{\A}(s)|}{D(t)/10}\Bigl)\Bigl) \frac{\Im((\wt{\A}(t)-\wt{\A}(s))^k)}{|\wt{\A}(t)-\wt{\A}(s)|^{k+1}}  \,dm_1(s) \\
   & + \int \Bigl(1 -\Psi\Bigl(\frac{|\wt{\A}(t)-\wt{\A}(s)|}{D(t)/10}\Bigl)\Bigl) \frac{\Im(\{\wt{\A}(t)-\wt{\A}(s)\}^k)}{|\wt{\A}(t)-\wt{\A}(s)|^{k+1}} (J(\wt{A})(s)-1) \,dm_1(s) \\
   & = S_1(x)+S_2(x)+S_3(x).
\end{align*}
Notice that
$$\| J(\wt{\A})-1\|_p \lesssim \|\A'\|^2_\infty.$$
Consequently, using the $L^p$ boundedness of $T^\perp$ on Lipschitz graphs (Theorem \ref{Huovbds}) we get 
$$\|S_3\|_p \leq \alpha^2. $$
Now we focus on $S_2$. First observe that, since $\alpha \ll 1$,
$$|t-s|\leq |\wt{\A}(t)-\wt{\A}(s)|\leq 2|t-s|.$$
Since $\Psi(z)=0$ if $|z|\leq 1/2$ and $\Psi(z)=1$ if $|z|\geq 1$, we deduce that 
$$\Psi\Bigl(\frac{t-s}{D(t)/10}\Bigl)-\Psi\Bigl(\frac{|\wt{\A}(t)-\wt{\A}(s)|}{D(t)/10}\Bigl)=0$$
if $|t-s|\leq D(t)/40$ or $|t-s|\geq D(t)/5$. Additionally, the mean value theorem ensures that 
$$\Bigl| \Psi\Bigl(\frac{t-s}{D(t)/10}\Bigl)-\Psi\Bigl(\frac{|\wt{\A}(t)-\wt{\A}(s)|}{D(t)/10}\Bigl)\Bigl|\leq \frac{C\alpha|t-s|}{D(t)}.$$
Consequently, 
\begin{align*}
    |S_2(x)|& \lesssim \int_{D(t)/40\leq |t-s|\leq D(t)/5} \frac{\alpha |t-s|}{D(t)} \frac{|\Im((\wt{A}(t)-\wt{A}(s))^k)|}{|t-s|^{k+1}} \,dm_1(s)     \lesssim \alpha^2,
\end{align*}
and therefore
$$\|S_2\|_p\lesssim \alpha^2.$$

We focus now on $S_1(x)$. Recall from (\ref{kernelexpand}) that
$$K^{\perp}_k(\wt{\A}(t)-\wt{\A}(s))=\sum_{\ell\in \mathbb{N}, \ell\text{ odd}}c_{k,\ell}\frac{(\A(t)-\A(s))^{\ell}}{(t-s)^{\ell+1}}.
$$
By the second order Taylor formula,
$$\A(t)-\A(s)=\A'(t)(t-s)+\tfrac{\A''(z)}{2}(t-s)^2 \text{ for some }z\in [t,s].
$$
For $s\in B(t, D(t)/5)$, we have that $D(z)\approx D(t)\approx D(s)$, and so the second derivative estimate given in part 1 of Proposition \ref{LipGraphProp} (and recalling the the definition of $\bdary$) yields that
$$\Bigl|\frac{\A''(z)}{2}(t-s)^2\Bigl |\lesssim \bdary\frac{(t-s)^2}{D(t)}.
$$
Now, employing the inequality $|(a+b)^{\ell} -a^{\ell}| \leq 2^{\ell}b\max(|a|, |b|)^{\ell-1}$ we arrive at
$$|(\A(t)-\A(s))^{\ell} - \A'(t)^{\ell}(t-s)^{\ell}|\leq C^{\ell}\alpha^{\ell-1}\bdary\frac{|t-s|^{\ell+1} }{D(t)},
$$
where we have used that $\frac{\bdary |t-s|}{D(t)}\leq \alpha$.  Next, we notice that for any $\kappa>0$,
$$\int_{\R\backslash B(t, \kap)}\Bigl[1-\Psi\Bigl(\frac{t-s}{D(t)}\Bigl)\Bigl]\frac{\A'(t)^{\ell}}{t-s}\,dm_1(s)=0,
$$
while  $\int_{\R}\Bigl[1-\Psi\Bigl(\frac{t-s}{D(t)}\Bigl)\Bigl]\,dm_1(s)\lesssim D(t)$.  Consequently, 
$$\Bigl|\int_{\R\backslash B(t,\kappa)}\Bigl[1-\Psi\Bigl(\frac{t-s}{D(t)}\Bigl)\Bigl]\frac{(\A(t)-\A(s))^{\ell}}{(t-s)^{\ell+1}}\,dm_1(s)\Bigl|\lesssim  C^{\ell}\alpha^{\ell-2}\bdary.
$$
Therefore, using (\ref{coefficientbds}), we have that since $\alpha \ll 1$,
\begin{equation}\begin{split}\nonumber|S_1(x)| & =\lim_{\kap\to 0}\Bigl|\int_{\R\backslash B(t,\kappa)} \Bigl(1 -\Psi\Bigl(\frac{t-s}{D(t)/10}\Bigl)\Bigl)K^{\perp}_k(\wt{\A}(t)-\wt{\A}(s))\,dm_1(s)\Bigl|\\
&\lesssim \sum_{\ell\in \mathbb{N}, \ell \text{ odd }}|c_{k, \ell}|C^{\ell}\bdary\alpha^{\ell-1}\lesssim \bdary\ll \alpha^2.
\end{split}\end{equation}

From here, and joining the previous estimates we conclude that
$$\| S\|_ {L^p(\Gamma\cap \pi^{-1}(4I_0))} \lesssim  \alpha^2.$$
The lemma follows.
\end{proof}

\subsection{The main comparison estimates}
Let $\wt\eta: [0,\infty) \to \mathbb{R}$ be a smooth non-increasing function with $\|\wt\eta\|_1=1/2$ such that $\supp \wt\eta \subset [0,1]$ and $\wt\eta$ equals $1$ on $[0,1/4]$. For $p>0$, we denote 
$$\eta_p(t)=\frac{1}{p}\wt\eta\left(\frac{|t|}{p}\right) \text{ for }t \in \mathbb{R}.$$
Therefore $\|\eta_p\|_{1}=1$.

We wish to show that $\sigma = \pi_{\#}\mu_{|\F}$ is close to a constant multiple of the Lebesgue measure, at least within $8I_0$.  In order to accomplish this, we introduce the function $g: \R \to \R$ given by
$$g(t)=\eta_{\sqrt{\bdary}D(t)} \ast \sigma.$$
Observe from (\ref{ProjGrowth}) that we can rudely estimate \begin{equation}\label{gup2}\|g\|_{\infty}\leq 3.\end{equation}  We will aim to prove more refined $L^p$ estimates on the function $g$, with (\ref{ProjGrowth}) our primary tool.

We will make use of the following elementary bound which appears as \cite{To3}, Lemma 10.3 (the proof merely uses of the fact that $D$ is a Lipschitz continuous function).
\begin{lemma} \label{Phi}
 For all $t,s \in \R$,
 \begin{equation*}
     |\eta_{\sqrt{\bdary}D(t)}(t-s)-\eta_{\sqrt{\bdary}D(s)}(t-s)|\lesssim \frac{\sqrt{\bdary}}{D(s)} \chi_{B(s,C\sqrt{\bdary}D(s))}(t).
 \end{equation*}
\end{lemma}

The next lemma is another estimate found in \cite{To3}, and is a simple consequence of (\ref{ProjGrowth}). (We recall that the proof of (\ref{ProjGrowth}) used properties of the transportation coefficients, and was necessarily quite different from the proof in \cite{To3}.)

\begin{lemma}\label{projlem}
 If $\varepsilon$ and $\theta$ have been chosen small enough with respect to $\alpha$, then we have 
 \begin{equation} \label{upg}
     0 \leq g(t)\leq 1+C\alpha^2 \text{ for all }t \in \R, 
 \end{equation}
and
 \begin{equation}\label{L2g}
     \|\chi_{8I_0}(g-1)\|_2\lesssim \alpha.
 \end{equation}
\end{lemma}
\begin{proof}  The lemma follows from integrating (\ref{ProjGrowth}).
For $t \in \R$, let $\psi: [0,\infty) \to \R$ be defined by $\psi(s)=\eta_{\sqrt{\bdary}D(t)}(s)$ and we denote $\sigma=\pi_{\#}(\mu_{|F})$. 

Observe that 
\begin{align*}
    g(t)  =- \int_{\sqrt{\bdary}D(t)/4}^{\sqrt{\bdary}D(t)} \sigma(B(t,r)) \psi'(r) \,dm_1(r),
\end{align*}
where  we have used that $\supp(\psi') \subset [\sqrt{\bdary}D(t)/4,\sqrt{\bdary}D(t)]$.
Consequently, since $\sqrt[4]{\eps}\ll\sqrt{\lambda}$, from (\ref{ProjGrowth}) (and that $\psi$ is monotone on $[0,\infty)$), we infer that
$$|g(t)|\leq (1+C\alpha^2) \int_{\sqrt{\bdary}D(t)/4}^{\sqrt{\bdary}D(t)}2r |\psi'(r)|\,dr\leq 1+C\alpha^2.$$
The inequality (\ref{upg}) is proved.  We next will show that
 \begin{equation}\label{L1g}
     \|\chi_{8I_0}(g-1)\|_1\lesssim \alpha^2.
 \end{equation} 
To this end, we will prove 
\begin{equation} \label{LowerG}
    \int_{8I_0} g(t) \,dm_1(t) \geq (1-C\sqrt{\bdary}) m_1 (8B_0\cap \R).
\end{equation}
To verify (\ref{LowerG}), first observe that since $D(t)\leq 9$ for all $t \in \pi(8B_0)$, we have 
\begin{align*}
    \int_{(8+9\sqrt{\bdary})I_0} g(t) \,dm_1(t)& =\int_{(8+9\sqrt{\bdary})I_0} \eta_{\sqrt{\bdary}D(t)} \ast\sigma(t) \,dm_1(t) \\
    & \geq \int_{8I_0} \int_{\R} \eta_{\sqrt{\bdary}D(t)}(t-s) \,dm_1(t) \,d\sigma(s).
\end{align*}
Using Lemma \ref{Phi}, 
$$|\eta_{\sqrt{\bdary}D(t)}(t-s)-\eta_{\sqrt{\bdary}D(s)}(t-s)| \lesssim \frac{1}{D(s)}\chi_{B(s,C\sqrt{\bdary}D(s))}(t).$$
Combining these two inequalities results in
\begin{align*}
    \int_{(8+9\sqrt{\bdary})I_0}g(t)\,dm_1(t) & \geq \int_{q \in 8I_0} \int_{\R} \eta_{\sqrt{\bdary}D(s)}(t-s)\,dm_1(t)\,d\sigma(s) \\
    &- \int_{8I_0} \frac{ m_1 (B(s,C\sqrt{\bdary}D(s)))}{D(s)} \,d\sigma(s) \\
    & \geq (1-C\sqrt{\bdary})\sigma(8I_0) \geq (1 - C\sqrt{\bdary}) m_1 (8I_0).
\end{align*}
In the final inequality we have used that $$\sigma(8I_0)\geq \mu(8B_0) = 16\cdot\delta_{\mu}(8B_0)\geq 16(1-C\sqrt{\eps})\delta_{\mu}(B_0) = 16(1-C\sqrt{\eps}),$$ where part (2) of Lemma \ref{densitycomparison} has been used in the inequality.  The inequality (\ref{LowerG}) now follows from the fact that $\|g\|_\infty \leq 3$ (recall (\ref{upg})).

But now, for suitable constant $C>0$,
\begin{align*}
\int_{8I_0}& |1+C\alpha^2-g(t)|\,dm_1(t)  \stackrel{(\ref{upg})}{=} \int_{8I_0}((1+C\alpha^2)-g(t))\,dm_1(t) \\
& \leq (1+C\alpha^2) m_1 (8I_0)-\int_{8I_0}g(t) \,dm_1(t)  \leq (C\alpha^2+ C \sqrt{\bdary}) m_1 (8B_0),
\end{align*}
and hence 
$$\int_{8I_0} |1-g(t)|\,dm_1(t) \leq (C \alpha^2 + C \sqrt{\bdary}) m_1 (8I_0),$$
achieving (\ref{L1g}) as $\bdary\ll \alpha^2$.
Finally, recalling (\ref{gup2}), 
$$\int_{8I_0} |1-g(t)|^2 \,dm_1(t) \leq (1+\|g\|_\infty)\int_{8I_0}|1-g(t)| \,dm_1(t) \lesssim \alpha^2$$
proving (\ref{L2g}).
\end{proof}

Going forward, will be convenient to make three definitions: 
\begin{definition*}   (1)  Denote by $P:\C\to\Gamma$ the mapping
$$P(x) = \wt{\A}(\pi(x)) \text{ for }x\in \C.
$$
(2)  Denote by $h:\Gamma\to \R$ the function  
$$h(x)=\frac{g(\pi(x))}{J\wt{\A}(\pi(x))}, \text{ for }x \in \Gamma. $$
(3) Define the Borel measure $\mtilde$ on $\C$ by
$$\mtilde = \mu_{|F},
$$
so that $\sigma$ is the pushforward of $\mtilde$ under the projection $\pi$.
\end{definition*}
From these definitions we have that, for a Borel set $E\subset 10I_0$ and a Borel function $f:\C\to \R$,
$$\int_{\pi^{-1}(E)}f\circ P \,d\mtilde = \int_E f\circ \wt{\A}\, d\sigma\text{ and }\int_E g\,dm_1 = \int_{\pi^{-1}(E)\cap \Gamma}h\,d\H^1.
$$
We will use these identities quite often in what follows.

\begin{lemma}\label{comparisonh} There is a constant $C>0$ such that for any $k\in [4,8]$, and any Borel measurable function $f:\C\to \R$.
\begin{equation}\begin{split}\Bigl|\int_{\pi^{-1}(kI_0)}fd(\mtilde - hd\H^1_{|\Gamma})\Bigl|&\lesssim\int_{kI_0}\bigl\{\osc_{B(\wt{\A}(t), C\sqrt{\bdary}\ell(t))}f\bigl\} \, d\sigma(t)\\
&+ \int_{(k+C\sqrt{\bdary})I_0\backslash k I_0} |f\circ \wt{\A}|\,dm_1\\
&+\int_{kI_0} (f\circ\wt{\A}) b\, dm_1, 
\end{split}\end{equation}
where $b:\R\to \R$ satisfies $\supp(b)\subset (k+1)I_0$ and $\|b\|_{\infty}\lesssim \sqrt{\bdary}$.
\end{lemma}

\begin{proof} Write
\begin{equation}\begin{split}\nonumber\int_{\pi^{-1}(kI_0)}& f \bigl[d\mtilde - hd\H^1_{|\Gamma})\bigl] \\
& = \int_{\pi^{-1}(kI_0)} (f - f\circ P) \,d\mtilde + \int_{kI_0} f\circ \wt{A}  \,(d\sigma - g\cdot dm_1).
\end{split}\end{equation}
Recall from (\ref{Fbdary}) that $\Gamma = \{\wt{\A}(t)\,:\,t\in \R\}$ and \begin{equation}\label{Fclose} F\subset \{x\in \C: \dist(x, \wt{A}(\pi(x)))\lesssim \bdary \ell(x)\}.\end{equation}
Therefore
$$\Bigl|\int_{\pi^{-1}(kI_0)} (f - f\circ P) \,d\mtilde\Bigl|\lesssim \int_{kI_0}\bigl\{\osc_{B(\wt{\A}(t), C\sqrt{\bdary}\ell(t))}f \bigl\}\, d\sigma(t).
$$
For the remaining term, we first observe that the function
$$g\chi_{kI_0}-[\eta_{\sqrt{\bdary}\ell(\,\cdot\,)}*(\chi_{kI_0}\sigma)] 
$$
is supported in $(k+C\sqrt{\bdary})I_0\backslash (k-C\sqrt{\bdary})I_0$,    and therefore
\begin{equation}\begin{split}\nonumber\Bigl|\int_{\R} &(f\circ \wt{A})\; \{g\chi_{kI_0}-[\eta_{\sqrt{\bdary}\ell(\,\cdot\,)}*(\chi_{kI_0}\sigma)] \}dm_1\Bigl|\\&\lesssim \int_{(k+C\sqrt{\bdary})I_0\backslash (k-C\sqrt{\bdary})I_0} |f\circ{\wt{A}}| \, g \, dm_1.
\end{split}\end{equation}
But now, using (\ref{upg}), and that $\ell(t)\gtrsim 1$ for $t\in kI_0\backslash (k-C\sqrt{\bdary})I_0$,
\begin{equation}\begin{split}\nonumber\int_{kI_0\backslash (k-C\sqrt{\bdary})I_0} |f\circ{\wt{\A}}| \, g \, \,dm_1&\lesssim  \int_{kI_0}\{\osc_{B(\wt{\A}(t), C\sqrt{\bdary}\ell(t))}f \}\,g(t)  \,dm_1(t)\\
&\;\;\;+ \int_{(k+C\sqrt{\bdary})I_0\backslash k I_0} |f\circ \wt{\A}|\,dm_1\\
&\lesssim \int_{kI_0}\{\osc_{B(\wt{\A}(t), C\sqrt{\bdary}\ell(t))}f \}\,  d\sigma(t) \\&\;\;\;+ \int_{(k+C\sqrt{\bdary})I_0\backslash kI_0} |f\circ \wt{\A}|\,dm_1.
\end{split}\end{equation}
It remains to consider
$$ \int_{kI_0} f\circ \wt{\A}  \,\bigl(d\sigma - \eta_{\sqrt{\bdary}\ell(\,\cdot\,)}*(\chi_{kI_0}\sigma)\,dm_1\bigl).
$$
First, using Fubini's theorem, observe that
\begin{equation}\begin{split}\nonumber \int_{\R} [f\circ \wt{\A}(t) &\eta_{\sqrt{\bdary}\ell(t)}*(\chi_{kI_0}\sigma)] \}\,dm_1(t) \\&=  \int_{kI_0}\int_{\R}\eta_{\sqrt{\bdary}\ell(s)}(t-s)(f\circ\wt{\A})(s)\,dm_1(s)d\sigma(t)
\end{split}\end{equation}
In order to obtain a convolution structure, we wish to replace $\ell(s)$ in the right hand integral with $\ell(t)$.  To this end, recall Lemma \ref{Phi}:
\begin{equation}\label{phidiffsimple}
|\eta_{\sqrt{\bdary}\ell(s)}(t-s) - \eta_{\sqrt{\bdary}\ell(t)}(t-s)|\lesssim \frac{\sqrt{\bdary}}{\ell(t)}\chi_{B(t, C\sqrt{\bdary}\ell(t))}.
\end{equation}
Crudely employing this bound, the difference
\begin{equation}\begin{split}\nonumber 
\int_{kI_0}\int_{\R}&\eta_{\sqrt{\bdary}\ell(s)}(t-s)(f\circ\wt{\A})(s)\,dm_1(s)d\sigma(t) \\&- \int_{kI_0}\int_{\R}\eta_{\sqrt{\bdary}\ell(t)}(t-s)(f\circ\wt{\A})(s)\,dm_1(s)d\sigma(t) 
\end{split}\end{equation}
can be bounded in absolute value by
$$\int_{(k+C\sqrt{\bdary})I_0} f\circ \wt\A(t)\Bigl\{\frac{\sqrt{\bdary}}{\ell(t)}\sigma(kI_0\cap B(t, C\sqrt{\bdary}\ell(t)))\Bigl\}\,dm_1(t).
$$
Labelling the function  in the brackets $\{\,\cdots\,\}$ appearing in this integral as $b$, we find from (\ref{ProjGrowth}) that $\|b\|_{\infty}\lesssim \sqrt{\bdary}$.

Finally, notice that
\begin{equation}\begin{split}\nonumber 
\Bigl|\int_{kI_0} & f\circ \wt{\A}\, d\sigma -\int_{kI_0}\int_{\R}\eta_{\sqrt{\bdary}\ell(t)}(t-s)(f\circ\wt{\A})(s)\,dm_1(s)d\sigma(y)\Bigl|\\
& \leq  \int_{kI_0}\int_{\R}\eta_{\sqrt{\bdary}\ell(t)}(t-s)|(f\circ \wt{\A}(t) - (f\circ\wt{\A})(s)|\,dm_1(s)d\sigma(t)\\
& \leq \int_{kI_0}\bigl\{\osc_{B(\wt{\A}(t), C\sqrt{\bdary}\ell(t))}f\bigl\} d\sigma(t).
\end{split}\end{equation}
The proof is complete.
\end{proof}

The following lemmas correspond  with Lemma 10.7 and Lemma 10.8 in Tolsa's paper (\cite{To3}).

\begin{lemma} \label{muToh} It holds that
$$ \|T^\perp_{\ell(\cdot),1}(\mu_{|F})-T^\perp_{\ell(\cdot),1}(h\mathcal{H}^1_{|\Gamma})\|_{L^2(\Gamma \cap \pi^{-1}(4I_0))}\lesssim \sqrt{\bdary}. $$
\end{lemma}

\begin{proof} For $t\in 4I_0$, the function
$$ y \mapsto K^{\perp}_{\ell(t), 1}(\wt{\A}(t)-y), \text{   }   y \in \C
$$
is supported on $6B_0\subset \pi^{-1}(6I_0)$.  We apply the comparison lemma Lemma \ref{comparisonh} with this function taking the place of $f$, and $k=6$.  Now, for $s \in 6I_0$,
\begin{equation}\begin{split}\nonumber\osc_{B(\wt{\A}(s), C\sqrt{\bdary}\ell(s))}(f) &\lesssim \sqrt{\bdary}\frac{\ell(s)}{\ell(t)^2+|\wt{\A}(s)-\wt{\A}(t)|^2}\\
&\lesssim\sqrt{\bdary}\frac{\ell(s)}{\ell(t)^2+|s-t|^2}\\
&\lesssim \sqrt{\bdary}\frac{\ell(s)}{\ell(s)^2+|s-t|^2},
\end{split}\end{equation}
where in the last inequality we have used that $\ell$ is a Lipschitz function, and so $\ell(s)\lesssim \ell(t)+|s-t|.$  We are thus led to estimate
\begin{equation}\begin{split}\nonumber\int_{4I_0}&\Bigl(\int_{6I_0}\frac{\ell(s)}{\ell(s)^2+|s-t|^2} d\sigma(s)\Bigl)^2 \,dm_1(t). 
\end{split}\end{equation}
To bound this integral we follow a standard path.  Observe that, for any $s\in \R$,
$$\int_{4I_0}\frac{\ell(s)}{\ell(s)^2+|t-s|^2}h(t)\,dm_1(t)\lesssim \mathcal{N}(h)(s),
$$
where $\mathcal{N}(f)(s):=\sup_{r>\ell(s)}\frac{1}{2r}\int_{B(s,r)}f \,dm_1.$  Since $\sigma(D(s, r))\lesssim r$ for any $r\geq \ell(s)$ (see property \ref{ProjGrowth} in Lemma \ref{sigmabd}), one verifies via the usual weak-type bound and interpolation\footnote{To be completely transparent we sketch the proof:  For $\lambda>0$, choose intervals $B_j = B(s_j, r_j)$ with $r_j\geq \ell(s_j)$ such that $B_j$ are disjoint, $\frac{1}{2r_j}\int_{B_j}f\,dm_1>\kap$, and $E_{\kap}:=\{\mathcal{N}(f)>\kap\}\subset \bigcup_j 3B_j$.  We arrive at the weak type bound  $\sigma(E_{\kap})\leq \sum_j \sigma (3B_j)\lesssim\sum_j m(B_j)\lesssim \frac{1}{\kap}\int f\,dm_1$, where in the second inequality it is used that $r_j\geq \ell(s_j)$.  Now, insofar as $\|\mathcal{N}f\|_{\infty}\leq \|f\|_{\infty}$, the subadditivity of $\mathcal{N}$ yields that $E_{\kap}\subset \{\mathcal{N}(f\chi_{\{|f|>\kap/2\}})>\kap/2$\}.  Therefore, applying the weak type bound to $f\chi_{\{|f|>\kap/2\}}$ yields that
$$\sigma(E_{\kap})\lesssim \frac{1}{\kap}\int_{\{f>\kap/2\}}f\,dm_1.
$$
The desired inequality follows from integrating both sides over $\kap$ with respect to the measure $\kap\chi_{(0,\infty)}\,dm_1(\kap)$.}  that
$\mathcal{N}: L^2(m_1)\to L^2(\sigma_{|6I_0})$ has operator norm $\lesssim 1$.  Duality therefore gives that
$$\int_{4I_0}\Bigl(\int_{6I_0}\frac{\ell(s)}{\ell(s)^2+|s-t|^2} d\sigma(s)\Bigl)^2 \,dm_1(t) \lesssim 1.$$
Regarding the remaining terms in the comparison lemma, the second term equals zero due to the compact support of $f$, so we need to bound \begin{equation}\begin{split}\nonumber\int_{4I_0}\Bigl|\int_{6I_0}K^{\perp}_{\ell(t), 1}(\wt{\A}(t)-\wt{\A}(s))b(s) \,dm_1(s)\Bigl|^2 \,dm_1(t). 
\end{split}\end{equation}
where $\|b\|_{2}\lesssim\sqrt{\bdary}$.  However, the operator boundedness of the Huovinen transform on Lipschitz graphs ensures that this double integral is $\lesssim \bdary$.
\end{proof}

Observe that, as a particular consequence of Lemma \ref{muToh} and part (1) of Theorem \ref{Huovbds}, we have that
\begin{equation}\label{mixedestimate}
\|T^\perp_{\ell(\cdot),1}(\mu_{|F})\|_{L^2(\Gamma \cap \pi^{-1}(4I_0))}\lesssim 1.
\end{equation}
In fact we can say this bound is of order $\alpha$ by further appealing to Lemma \ref{GammatoStop}, but this gain will not be of use.

\begin{lemma} \label{outergammatomu}
$$\|T^\perp_{\ell(\cdot),1}(\mu_{|F})\|^2_{L^2(hd\mathcal{H}^1_{|\Gamma \cap \pi^{-1}(4I_0)})} - \|T^\perp_{\ell(\cdot),1}(\mu_{|F})\|^2_{L^2(\mu_{|F \cap \pi^{-1}4I_0})}\lesssim (M+1)\sqrt{\bdary}.$$
\end{lemma}

\begin{proof}
We apply the comparison estimate (Lemma \ref{comparisonh}) with $k=4$ and $f = |T^\perp_{\ell(\cdot),1}(\mu_{|\F})|^2$.  Now, 
$$\osc_{B(\wt{\A}(x), C\sqrt{\bdary}\ell(x))}|T^\perp_{\ell(\,\cdot\,),1}(\mu_{|\F})|\lesssim \sqrt{\bdary}.
$$
Therefore, 
$$\osc_{B(\wt{\A}(x), C\sqrt{\bdary}\ell(x))}|T^\perp_{\ell(\,\cdot\,),1}(\mu_{|\F})|^2\lesssim \sqrt{\bdary}\inf_{B(\wt{\A}(x), C\sqrt{\bdary}\ell(x))}|T^\perp_{\ell(\,\cdot\,),1}(\mu_{|\F})|+\bdary.
$$
But now, using (\ref{Fbdary}) once again, we have
$$\int_{4I_0}\inf_{B(\wt{\A}(\, \cdot\,), C\sqrt{\bdary}\ell(\, \cdot\,))}|T^\perp_{\ell(\,\cdot\,),1}(\mu_{|\F})|d\sigma \lesssim \int_{5B_0}|T^\perp_{\ell(\,\cdot\,),1}(\mu_{|\F})|d\mu_{|F}
$$
and the term on the right hand side is $\lesssim M$ due to the operator boundedness of the Huovinen transform on $L^2(\mu)$ (assumption (d) in Main Lemma \ref{mainlemma}).

 For the second term appearing in the comparison estimate, observe that since $\ell(t)\gtrsim 1$ for $t\notin 4I_0$, 
\begin{equation}\label{smalloutside}|(f\circ\wt{\A})(t)|\lesssim 1\text{ for }t\notin 4I_0\end{equation} and therefore
$$\int_{(k+C\sqrt{\bdary})I_0\backslash k I_0} |f\circ \wt{\A}|\,dm_1\lesssim \sqrt{\bdary}.
$$
Finally, for the third term appearing in Lemma \ref{comparisonh}, recall that $\|b\|_{\infty}\lesssim\sqrt{\bdary}$, and therefore
$$\int_{5I_0} |T^{\perp}_{\ell(t), 1}(\mu_{|\F})(\wt{\A}(t))|^2b(t) \,dm_1(t)\lesssim\sqrt{\bdary} \|f\|^2_{L^2(\Gamma \cap \pi^{-1}(5I_0))}.
$$
We split $ \|f\|^2_{L^2(\Gamma \cap \pi^{-1}(5I_0))} =  \|f\|^2_{L^2(\Gamma \cap \pi^{-1}(4I_0))}+  \|f\|^2_{L^2(\Gamma \cap \pi^{-1}(5I_0\backslash 4I_0))}.$
The first term is controlled by (\ref{mixedestimate}), while the second is controlled by (\ref{smalloutside}).
\end{proof}

The final step required to prove Proposition \ref{Tmulowbd} is the following lemma

\begin{lemma} \label{curvToh}
We have 
$$\Bigl|\|T^\perp_{\ell(\cdot),1}(h \mathcal{H}^1_{|\Gamma})\|^2_{L^2(hd\mathcal{H}^1_{|\Gamma \cap \pi^{-1}(4I_0)})}-\| T^\perp_{\ell(\cdot),1}( \mathcal{H}^1_{|\Gamma})\|^2_{L^2(\Gamma \cap \pi^{-1}(4I_0))}\Bigl|\lesssim \alpha^2.$$
\end{lemma}

\begin{proof}  First observe that, as a consequence of (\ref{L2g}) in Lemma \ref{projlem},
\begin{equation}\label{hcloseto1}\|(h-1)\|_{L^2(\Gamma\cap \pi^{-1}(6I_0))}\lesssim \alpha^2,\end{equation} so
\begin{equation}\begin{split}\nonumber \| &T^\perp_{\ell(\cdot),1}(h \mathcal{H}^1_{|\Gamma}) - T^\perp_{\ell(\cdot),1}( \mathcal{H}^1_{|\Gamma})\|_{L^2(hd\mathcal{H}^1_{|\Gamma \cap \pi^{-1}(4I_0)})}\\
&\lesssim \|T^{\perp}_{\ell(\cdot),1}((h-1) \mathcal{H}^1_{|\Gamma})\|_{L^2(\Gamma \cap \pi^{-1}(4I_0))}\\
&\lesssim  \|T^{\perp}_{\ell(\cdot),1}((h-1) \mathcal{H}^1_{|\Gamma\cap \pi^{-1}(6I_0)})\|_{L^2(\Gamma \cap \pi^{-1}(4I_0))}\\&\lesssim \|h-1\|_{L^2(\Gamma\cap \pi^{-1}(6I_0))}\lesssim \alpha^2.
\end{split}\end{equation}
Secondly, 
\begin{equation}\begin{split}\nonumber \Bigl|\|T^{\perp}_{\ell(\cdot), 1}(&\mathcal{H}^1_{|\Gamma})\|^2_{L^2(hd\mathcal{H}^1_{|\Gamma \cap \pi^{-1}(4I_0)})} - \|T^{\perp}_{\ell(\cdot), 1}(\mathcal{H}^1_{|\Gamma})\|^2_{L^2(\Gamma \cap \pi^{-1}(4I_0))}\Bigl|\\
& =\Bigl|\int |T^{\perp}_{\ell(\cdot), 1}(\mathcal{H}^1_{|\Gamma})|^2 (h-1) d\mathcal{H}^1_{\Gamma \cap \pi^{-1}(4I_0))}\Bigl|\\
&\leq \|T^{\perp}_{\ell(\cdot), 1}(\mathcal{H}^1_{|\Gamma})\|^2_{L^4(\Gamma \cap \pi^{-1}(4I_0))}\|(h-1)\|_{L^2(\Gamma \cap \pi^{-1}(4I_0))}
\end{split}\end{equation}
Appealing to Lemma \ref{GammatoStop} and part (1) from Theorem \ref{Huovbds} (with $p=4$), we get
$$ \|T^{\perp}_{\ell(\cdot), 1}(\mathcal{H}^1_{|\Gamma})\|^2_{L^4(\Gamma \cap \pi^{-1}(4I_0))}\lesssim \alpha^2.
$$

On the other hand, from (\ref{hcloseto1}), $\|(h-1)\|_{L^2(\Gamma \cap \pi^{-1}(4I_0))}\lesssim \alpha^2$.  
Therefore, 
\begin{equation}\begin{split}\nonumber \Bigl|\|T^{\perp}_{\ell(\,\cdot\,), 1}(&\mathcal{H}^1_{|\Gamma})\|^2_{L^2(hd\mathcal{H}^1_{|\Gamma \cap \pi^{-1}(4I_0)})} - \|T^{\perp}_{\ell(\,\cdot\,), 1}(\mathcal{H}^1_{|\Gamma})\|^2_{L^2(\Gamma \cap \pi^{-1}(4I_0))}\Big|\lesssim \alpha^4,
\end{split}\end{equation}
and the lemma follows.
\end{proof}

\begin{proof}[Proof of Proposition \ref{Tmulowbd}] Notice that employing (\ref{gradL2}), followed by applying  Lemma \ref{GammatoStop}, then Lemma \ref{curvToh}, Lemma \ref{muToh} (observing that $h\lesssim 1$ on $\Gamma$ as a consequence of (\ref{upg})), and then finally Lemma \ref{outergammatomu}, leads to the following estimate
$$\|T^{\perp}_{\ell(\cdot), 1}(\mu_{|\F})\|_{L^2(\mu_{|F\cap \pi^{-1}(4I_0)})}\gtrsim \|\A'\|_{L^2(\R)}-C\alpha^{2}.
 $$
 Since $\mu(10B_0\backslash \F)$ is small, the proposition follows from $L^2(\mu)$ boundedness of the Huovinen transform (see (a) and (d) from Main Lemma \ref{mainlemma}).
\end{proof}

 \section{The final contradiction:  The proof of Proposition \ref{F3small}}
   
\begin{proof}[Proof of Proposition \ref{F3small}]
Assume that $\mu(F_2)>\alpha^{1/2}$.  Then by Lemma \ref{BigGrad}, 
$$\|A'\|_{L^2(\R)}^2\gtrsim \alpha^{5/2}.
$$
Therefore, Proposition \ref{Tmulowbd} yields that
$$\|T^\perp_{\ell(\cdot),1}(\mu)\|^2_{L^2(\mu_{|\F\cap \pi^{-1}(4I_0)})}\gtrsim \alpha^{5/2}.$$
This contradicts assumption (e) of the Main Lemma.
\end{proof}

\appendix
\section{Continuity of the transportation coefficients}

In this appendix we prove Lemma \ref{continuity}.\\

We start with a simple remark.
 \begin{remark} Given two pairs $(x,r), (x_1,r_1) \in \C \times \R^+,$ we define the map
$$ O(y)= \frac{y-x_1}{r_1} r +x. $$
The map $O$ satisfies $O(B(x_1,4r_1))=B(x,4r)$ and $\|O\|_{Lip}= \frac{r}{r_1}=(\|O^{-1}\|_{Lip})^{-1}$.
Moreover, $O$ establishes a bijection between $\mathcal{F}_{x,r}$ and $\mathcal{F}_{x_1,r_1}.$ Given $f \in \mathcal{F}_{x,r}$ we will denote by $f_O(\cdot)=(f \circ O)(\cdot) \in \mathcal{F}_{x_1,r_1}.$

Below, given a sequence $\{(x_j,r_j)\}_{j\geq 1}$ relative to $(x,r)$, we will denote by $O_j:=O$ the function corresponding to the pairs $(x,r)$ and $(x_j,r_j)$. Furthermore, we will write $f_j$ in place of $f_O$.
\end{remark}
\begin{lemma}(Continuity of transportation coefficients)
Given a sequence $\{(x_j,r_j)\}_{j \geq 0} \in \C \times \R^+$ satisfying that $x_j \to x_0 \in \C$ and $r_j \to r_0$, we have the following:
\begin{enumerate}
    \item $\alpha_\mu(B(x_j,r_j) \to \alpha_\mu(B(x_0,r_0)) $.
    \item Given a sequence $D_j \in \G_{x_j}$ for all $j \geq 0$ satisfying $\angle(D_j,D_0) \to 0$, then $\alpha_{\mu,\mathcal{H}^1_{|D_j}}(B(x_j,r_j)) \to \alpha_{\mu,\mathcal{H}^1_{|D}}(B(x_0,r_0)).$
\end{enumerate}
\end{lemma}
\begin{proof}

With this in mind, both parts (1) and (2) of the lemma are consequences of the following statement:
 \emph{If $x_j 
 \to x_0$ and $r_j \to r_0$, then for every $\eta>0$, we can find $\delta >0$ and $j_0 \geq 1$ such that for every $j \geq j_0$ we have that:}
\begin{equation}\begin{split}\label{alphacont}
  \sup_{\substack{(D,D') \in \G_{x_0} \times \G_{x_j}, \\ \angle(D,D')\leq \delta}} \!\!\!\!|\alpha_{\mu,D}(B(x_0,r_0))\!-\!\alpha_{\mu,D'}(B(x_j,r_j))| \!\leq\! \eta \delta_\mu(B(x_0,5r_0)). 
\end{split}  \end{equation}

We focus on proving (\ref{alphacont}). Fix $\eta>0$. We will use the following two facts, which are routinely verified:
\begin{enumerate}
    \item  There exists $j_0 \geq 1$ such that for every $j \geq j_0$ we have that 
    $$\Bigl|\int \varphi_{x_j,r_j} f_j \,d\mu - \int \varphi_{x_0,r_0} f \,d \mu \Bigl| \leq \eta \mu(B(x_0,5r_0)),$$
    for every $f \in \mathcal{F}_{x_0,r_0}$.  (For this one needs to observe that the collection $\mathcal{F}_{x_0,r_0}$ is relatively compact in the uniform topology.)
\item  There exists $\delta >0$ and $j_0 \geq 1$ such that for every $j \geq j_0$ and for any two lines $D \in \G_x$ and $ D' \in \G_{x_j}$  satisfying $\angle(D,D') \leq \delta$,
$$ \left| \frac{1}{r_0}\int \varphi_{x_0,r_0} f \,d\mathcal{H}^1_{|D} - \frac{1}{r_j} \int \varphi_{x_j,r_j} f_j  \,d\mathcal{H}^1_{|D'} \right| \leq  \eta,$$
for every $f \in \mathcal{F}_{x_0,r_0}$.
\end{enumerate}

Given $D \in \G_{x_0}$ and $ D' \in \G_{x_j}$, we  put 
$$c= \frac{1}{\int \varphi_{x_0,r_0} \,d\mathcal{H}^1_{|D}} \int \varphi_{x_0,r_0} \,d\mu, \text{ and } c_j=\frac{1}{\int \varphi_{x_j,r_j} \,d\mathcal{H}^1_{|D'}} \int \varphi_{x_j,r_j} \,d\mu.$$
Since the denominators in the fractions appearing in $c$ and $c_j$ coincide, we can choose $j_0$ larger if necessary to ensure that $| c_j -  c|\lesssim \eta \delta_{\mu}(B(x_0, 5r_0))$ for all $j\geq j_0$.  Together with fact (2), this remark ensures that given any two lines $D \in \G_x$ and $D' \in \G_{x_j}$ with $\angle(D,D') \leq \delta$ and $j \geq j_0$ we have that
$$ \left| \frac{c}{r_0}\int \varphi_{x_0,r_0} f \,d\mathcal{H}^1_{|D} - \frac{c_j}{r_j} \int \varphi_{x_j,r_j} f_{j}  \,d\mathcal{H}^1_{|D'} \right| \lesssim \eta \delta_\mu(B(x_0,5r_0)), $$
for every $f \in \mathcal{F}_{x_0,r_0}$.  Combining the previous inequality with fact (1) above yields that if $\angle(D,D') \leq \delta$ and $j \geq j_0$, then for any $f \in \mathcal{F}_{x_0,r_0}$,
\begin{align*}\Bigl| \frac{1}{r_0}\int \varphi_{x_0,r_0} f \,d(\mu-c\mathcal{H}^1_{|D}) & -    \frac{1}{r_j}\int \varphi_{x_j,r_j} f_j \,d(\mu-c_j\mathcal{H}^1_{|D'}) \Bigl| \\&\lesssim \eta \delta_\mu(B(x_0,5r_0)).\end{align*}
%Consequently, we have that 
%\begin{align*}&\left| \left|\frac{1}{r}\int \varphi_{x,r} f \,d(\mu-c\mathcal{H}^1_{|D+x})\right| -    \left|\frac{1}{r_j}\int \varphi_{x_j,r_j} f_j \,d(\mu-c_j\mathcal{H}^1_{D'+x_j})\right| \right| \\
%&\lesssim \eta \delta_\mu(B(x,5r)).\end{align*}
The claimed estimate (\ref{alphacont}) now follows from (several applications of) the triangle inequality.
\end{proof}

\iffalse Indeed, 
\begin{align*}
& \left|\frac{1}{r}\int \varphi_{x,r} f \,d(\mu-c\mathcal{H}^1_{|D+x})\right| -  \sup_{g \in \mathcal{F}_{x_j,r_j}}   \left|\frac{1}{r_j}\int \varphi_{x_j,r_j} g \,d(\mu-c_j\mathcal{H}^1_{|D'+x_j})\right| \\
& \leq  \left|\frac{1}{r}\int \varphi_{x,r} f \,d(\mu-c\mathcal{H}^1_{|D+x})\right| -    \left|\frac{1}{r_j}\int \varphi_{x_j,r_j} f_j \,d(\mu-c_j\mathcal{H}^1_{|D'+x_j})\right| \lesssim \eta \delta_\mu(B(x,5r)).
\end{align*}
Since $f \in \mathcal{F}_{x,r}$ is arbitrary, we have that 
$$\sup_{f \in \mathcal{F}_{x,r}}\left|\frac{1}{r}\int \varphi_{x,r} f \,d(\mu-c\mathcal{H}^1_{|D+x})\right| -  \sup_{g \in \mathcal{F}_{x_j,r_j}}    \left|\frac{1}{r_j}\int \varphi_{x_j,r_j} g \,d(\mu-c_j\mathcal{H}^1_{|D'+x_j})\right| \lesssim \eta \delta_\mu(B(x,5r).$$
Arguing in the same way, we obtain then 
$$\left|\sup_{f \in \mathcal{F}_{x,r}}\left|\frac{1}{r}\int \varphi_{x,r} f \,d(\mu-c\mathcal{H}^1_{|D+x})\right| -  \sup_{g \in \mathcal{F}_{x_j,r_j}}    \left|\frac{1}{r_j}\int \varphi_{x_j,r_j} g \,d(\mu-c_j\mathcal{H}^1_{|D'+x_j})\right| \right| \lesssim \eta \delta_\mu(B(x,5r).$$
This gives us (\ref{alphacont}) and completes the proof of the lemma.\fi

\section{The proof of Proposition \ref{LipGraphProp}}\label{LipConstructAp}

This appendix gives a detailed proof of Proposition \ref{LipGraphProp}.  We follow \cite{L} quite closely.

\subsection{Constructing the map $A$ on $\pi(\Z)$}

\begin{lemma} \label{Slope}
Let $(x,t_1),(y,t_2) \in S$ be such that 
$$|x-y| \geq \sqrt{\bdary}\max(t_1,t_2).$$ Then 
\begin{equation} \label{conclusion}
    |\pi^\perp(x) - \pi^\perp(y)|\leq \bigl(\alpha+C\sqrt{\bdary}\bigl)|\pi(x)-\pi(y)|.
\end{equation}
\end{lemma}

\begin{proof} Put $r= \min(\bdary^{-1/2}|x-y|,10)$.  Then $(x,r)\in S$ and $\pi(y)\in \pi(B(x, 2|x-y|)$), so we infer from Lemma \ref{CorStrip} that
$$\dist(y,D)\lesssim\sqrt{\lambda} |x-y|$$ for some line $D \in \G_x$  with $\angle(D,D_0)\leq \alpha$.  In particular, if $Y_D$ denotes the projection of $Y$ onto $D$, then
\begin{equation}\label{YDsmallpert}\bigl(1+C\sqrt{\bdary}\bigl)|x-y|\geq |x-y_D|\geq \bigl(1-C\sqrt{\bdary}\bigl)|x-y|.
\end{equation}
But, since $\angle(D, D_0)\leq \alpha$,
$$|\pi^{\perp}(x)- \pi^{\perp}(y_D)|\leq \alpha |\pi(x)- \pi(y_D)|.
$$
Projections contract distances, so we conclude from (\ref{YDsmallpert}) that
$$|\pi^{\perp}(x)- \pi^{\perp}(y)|\leq \alpha |\pi(x)- \pi(y)|+ C\sqrt{\bdary}|x-y|.
$$
Finally, since $|x-y|\leq |\pi^{\perp}(x)- \pi^{\perp}(y)|+ |\pi(x)- \pi(y)|$, we arrive at the desired statement after noting that $\bdary \ll1$.
 \end{proof}

\begin{cor}\label{projection}
Suppose $x,y \in \C$ and $t\geq 0$ are such that $$|\pi(x)-\pi(y)| \leq t,\; d(x)\leq t,\text{ and } d(y)\leq t.$$ Then $|x-y|\lesssim t$.
\end{cor}

\begin{proof} We may assume $|x-y|>3t$ since otherwise there is nothing to prove.  By definition we can find $(X,s_1)$ and $(Y, s_2)$ belonging to $S$, with $|x-X|+s_1\leq t$ and $|y-Y|+s_2\leq t$.  But then $(X,t)$ and $(Y, t)$ both belong to $S$ and $|X-Y|>t$.  Therefore, Lemma \ref{Slope} yields that
\begin{equation}\label{projperpeq} |\pi^\perp(X)-\pi^\perp(Y)|\lesssim |\pi(X)-\pi(Y)|.
\end{equation}
But by the triangle inequality, $|\pi(X)-\pi(Y))|\leq 3t$, and therefore from (\ref{projperpeq}) we infer that  $ |\pi^\perp(X)-\pi^\perp(Y)|\lesssim t$.  Appealing to the triangle inequality again we conclude that
$$|\pi^\perp(x)-\pi^\perp(y)|\lesssim t.
$$
Given that we are assuming that $|\pi(x)-\pi(y)| \leq t$, the corollary follows.
\end{proof}

 \begin{cor} \label{ASlope}
 Let $x,y \in \Z$. Then 
 $$|\pi^\perp(x) - \pi^\perp(y)| \leq 2\alpha |\pi(x)-\pi(y)|.$$
 \end{cor}
 \begin{proof} Assume $x\neq y$.  Given $t\in (0, |x-y|)$, we can find pairs $(X,t)$ and $(Y,t) \in S$ where $X$ and $Y$ are arbitrarily close to $x$ and $y$ respectively, and $d(X,Y)>t$.  Since $\sqrt{\bdary}\ll \alpha$, Lemma \ref{Slope} now yields that 
 $$|\pi^\perp(X)-\pi^\perp(Y)|\leq 2\alpha |\pi(X)-\pi(Y)|,$$
and the statement follows since projections are continuous.
 \end{proof}

Define the function $A$ on $\pi(\Z)$ by setting $$A(\pi(x))= \pip(x) \text{ for } x \in \Z.$$

 Keeping in mind Corollary \ref{ASlope}, we see that $A$ is well defined on $\pi(\Z)$, and moreover,  $A:\pi(\Z) \to D^\perp_0$ is $2\alpha$-Lipschitz:
\begin{equation}\label{piZLip}|A(\pi(x)) - A(\pi(y))| \leq 2\alpha |\pi(x) - \pi(y)|.\end{equation}

\subsection{Extending $A$ over $D_0$}  We now select a Whitney cover relative to the function $D$.  Set $\I$ to be a collection of dyadic intervals in $\R$.

 For $p \in 10I_0\backslash \pi(\Z)$ we have $D(p)>0$.  Set $I_p$ to be the largest dyadic interval in $\I$ containing $p$ satisfying
 $$\diam I_p \leq \frac{1}{20}\inf_{u \in I_p} D(u).$$
 The interval $I_p$ exists because $D(p)>0$ and $D$ is Lipschitz.
 
Consider the collection of these intervals and relabel them $\I_{\max}=\{I_j\}_j$. The intervals $I_j$ are disjoint and the collection of $2I_j$ is a cover of $10I_0\setminus \pi(\Z)$. 
 
 The following lemma collects standard properties regarding this collection of intervals and follows immediately from the definitions ( and using that $D$ is $1$-Lipschitz), see \cite{L} page 847 or \cite{To5}, page 248.
 
 \begin{lemma} \label{dyadic} The following assertions hold.
 \begin{enumerate}
 \item If $p \in 10I_j$ then $10\diam I_j \leq D(p) \leq 60 \diam I_j$.
     \item Whenever 
     $10 I_i \cap 10 I_j \neq \emptyset$, then 
     $$ \diam I_j \lesssim \diam I_i \lesssim \diam I_j.$$
     \item There exists $N>0$ (an absolute constant) such that for every $i$, at most $N$ intervals $I_j$ satisfy $10I_i \cap 10I_j \neq \emptyset$.
 \end{enumerate}
 \end{lemma}

\begin{lemma} \label{SizeBall}
 For any $I_i \in \I_{\max}$, there exist a ball $B_i \in S$ such that 
  \begin{enumerate}
     \item $\diam I_i \leq r(B_i) \lesssim \diam I_i,$ 
           \item $d(\pi(c(B_i)),I_i) \leq 120 \diam I_i,$ and
      \item $d(\pi(B_i),I_i) \lesssim \diam I_i.$
  \end{enumerate}
  \end{lemma}

  \begin{proof}
  Let $p \in I_i$. We can find $(x,t)\in S$ such that $d(p,\pi(x)) + t \leq 2D(p) \leq 120 \diam I_i$ (see part (1) of Lemma \ref{dyadic}). The ball $B(x, \max\{t, \diam(I_i)\})$ satisfies properties (1) and (2), from which (3) immediately follows.
  \end{proof}

\begin{definition*}[The function $A_i$] For each of the balls $B_i\in S$, we set $D_i \in \G_{c(B_i)}$ to be such that $\alpha_{\mu, D_{i}}(B_i)\leq \eps$ with $\angle (D_i, D_0)\leq \alpha$.

Put $A_i$ to be the affine function $A_i: D_0 \to D^\perp_0$ whose graph is $D_i=D_{B_i}$.  Then certainly $A_i$ is Lipschitz of constant $\leq 2\alpha$.
\end{definition*}

  \begin{lemma} \label{Bi}
  Whenever $10 I_i \cap 10I_j \neq \emptyset$,
  \begin{enumerate}
      \item $d(B_i,B_j)\lesssim \diam I_j$,
      \item for any $L>1$, $|A_i(q)-A_j(q)| \lesssim  L^2 \bdary \diam I_j$ for any $q \in L I_j$, and
      \item $\dist(D_i, D_j\cap B_j)\lesssim  \bdary \diam(I_i)$ and $|(A_i-A_j)'| \lesssim \bdary.$
  \end{enumerate}
  \end{lemma}
  
  \begin{proof}
  For (1) we apply Corollary \ref{projection}:  If $10 I_i \cap 10I_j \neq \emptyset$, then $\ell(R_i)\approx \ell(R_j)$, so $\ell(B_i)\approx \ell(B_j)$ (see part (1) of Lemma \ref{SizeBall}).  But then 
  $d(\pi(B_i), \pi(B_j))\lesssim \ell(I_i)$ (see part (3) of Lemma \ref{SizeBall}).  Therefore applying Corollary \ref{projection} with $x,y$ to be the centers of $B_i$ and $B_j$, and $t$ a suitable constant multiple of $\diam(I_j)$ yields the required inequality.
  
Given (1), the balls $B_i, B_j$ both lie in $S$, satisfy $r(B_i)\approx r(B_j)$ and $CB_i\cap CB_j\neq \varnothing$, for some absolute constant $C>0$.  We may therefore infer from Lemma \ref{notmuchturn} that
\begin{equation}\label{ijlines}\dist(y, D_{i})\lesssim \bdary \cdot r(B_i)\text{ for all }y\in D_j\cap B_j,
\end{equation}
from which property (3) is an immediate consequence (recalling that $r(B_i)\approx r(B_j)$).  

Finally, since $D_i$ and $D_j$ both form an angle $\leq \alpha$ with $D_0$, statement (2) also follows from (\ref{ijlines}) and parts (1) and (3) of Lemma \ref{SizeBall}.  \end{proof}
   
   \begin{lemma} \label{G}
There exists $C>1$ such that if $x\in F\backslash \Z$ then   $\pi(x)\in 3I_i$  and $x\in CB_i$ for some $I_i \in \I_{\max}$.  \end{lemma}
  
  \begin{proof} Let $x \in F \setminus \Z$.
  We have that either 
  \begin{enumerate}
      \item $\pi(x) \in \pi(\Z)$ and there exists $y \in \Z$ such that $\pi(y)=\pi(x),$
      \item or $\pi(x) \in 3I_i$ for some $i$ and by part $(3)$ of Lemma \ref{SizeBall}, there exists $C>1$ such that $\pi(x) \in \pi(CB_i). $
  \end{enumerate}  
  We first will rule out that possibility (1) can occur.  To this end, we notice that $B(y, 2|x-y|)$ belongs to $S$, and $x\in \pi(B)$, so by Lemma \ref{CorStrip}, $\dist(x, D)\lesssim \bdary |x-y|$ where $D\in \G_y$ satisfies $\angle (D, D_0)\leq \alpha$.  On the other hand, $\pi(x)=\pi(y)$ and so $|x-y|\lesssim \dist(x,D)$, which is absurd given that $\bdary \ll 1$.
  
  We may therefore assume that (2) holds.   Then $CB_i\in S$ and $\pi(x)\in \pi(CB_i)$.  Therefore, Lemma \ref{CorStrip} ensures that $x\in 3CB_i$, and the proof is complete.  \end{proof}

 The previous lemma has the following useful consequence.
 
  \begin{cor} \label{Dcompar}
  For any $x \in F$,
  $$d(x) \lesssim D(\pi(x)) \leq d(x).$$
  \end{cor}
  \begin{proof}
  If $d(x)=0$, the conclusion is obvious.   Otherwise, $\pi(x) \in 3I_i$ for some $i$.  It follows from the definition of the intervals $I_i$, and Lemmas \ref{dyadic} and \ref{SizeBall} that $D(\pi(x)) \gtrsim  r(B_i)$. On the other hand, by Lemma \ref{G}) $x \in CB_i$, and so $d(x)\lesssim r(B_i)$.
  \end{proof}

    \begin{definition*}Choose a partition of unity $\psi_i$ subordinate to the cover $(2I_i)_i$ satisfying  $$ \| \psi_i'\|_{\infty} \lesssim \frac{1}{\diam I_i}, \text{ and }\| \psi''_i\|_{\infty} \lesssim\frac{1}{(\diam I_i)^2}.$$
    For $p \in D_0\backslash \pi(\Z)$, we define $A$ as
  $$A= \sum_{I_i \in \mathcal{I}_{\max}} \psi_i\cdot A_i,
$$
  \end{definition*}

 \begin{lemma}\label{unitscaleAclosetoD0} For every $p\in 3I_0$,
 $$|A(p)|\lesssim \bdary.
 $$
 \end{lemma}

\begin{proof} From Corollary \ref{CorStripUnit}, we have
$$F\cap \pi^{-1}(10I_0)\subset \Bigl\{\dist(\cdot, D_0)\lesssim \bdary \Bigl\}.$$
Since $\Z\subset F$, we may assume that $(p, A(p))\notin \Z$. We want to prove that $A(p) = \sum_k \psi_k(p)A_k(p)$ satisfies $|A(p)|\lesssim \bdary.$ Since $\sum_k \psi_k(p) \leq 1$, it suffices to prove that
$$|A_k(p)|\lesssim \bdary \text{ whenever }\psi_k(p)\neq 0.
$$ 
Fix such a $k$.  Consider the ball $B_k$ for which $D_k$ is the graph of $A_k$, then $(p, A_k(p))\in CB_k$ for some $C>0$ (Lemma \ref{SizeBall}).  But now we may apply Corollary \ref{stopclosetoB0} to find that $|A_k(p)| = \dist((p, A_k(p)), D_0)\lesssim \bdary$.
 \end{proof}

  \begin{lemma}\label{ALipschitz}
   $A:3I_0\to D_0^{\perp}$ is a $C\alpha$-Lipschitz function.
  \end{lemma}

  \begin{proof}
  Fix $p,q\in 3I_0$.

  If $p, q\in \pi(\Z)$ this has already been proved (recall (\ref{piZLip}), so we will assume that $p\notin\pi(\Z)$, and so $p\in 2I_i$ for some $i$.

  First suppose that $q\notin \pi(\Z)$, so $q\in 2I_k$ for some $k$, and $\sum_k \psi_k(p)=\sum_k \psi_k(q)=1$. 
  
 \textbf{Case 1:} $q\in 1000I_i$.  Then write 
   \begin{equation}\begin{split}\nonumber
   |A(p)-A(q)|  &
   \leq \sum_{j} \psi_j(p)|A_j(p)-A_j(q)]| \\
&  +\sum_j |\psi_j(p)-\psi_j(q)||A_j( q)-A_k(q)|\\
  \end{split}\end{equation}
  The first term is bounded by $2\alpha |p-q|$.  For the second term, we infer from part (2) of Lemma \ref{Bi} (and Lemmas \ref{dyadic}, \ref{SizeBall}) that for any $j$ where $\psi_j(p)$ or $\psi_j(q)$ is non-zero,
  $$|A_j(q)-A_k(q)|\lesssim \bdary \ell(R_j).
  $$
  On the other hand, $|\psi_j(p)-\psi_j(q)|\lesssim \frac{1}{\ell(I_j)}$, so, insofar as the number of $j$ with either $\psi_j(p)$ or $\psi_j(q)$ is non-zero is bounded by an absolute constant,
  \begin{equation}\begin{split}\label{fixone}
  |A(p)-A(q)| \leq 2\alpha |p-q|+ C\bdary|p-q|\leq 3\alpha |p-q|.
  \end{split}\end{equation}
    
\textbf{Case 2:} $q\notin 1000I_i$.   Then $|p-q|\gtrsim \max\{\diam(I_i), \diam(I_k)\}$.     Consider the pair $x= c(B_{k})$ and $t= r(B_k)$.   Then $|x- c(B_i)|\gtrsim \max(r(B_i), t)\gg \sqrt{\bdary}\max(r(B_i), t)$, and so Lemma \ref{Slope} yields that
\begin{equation}\begin{split}|A_{k}(\pi(c(B_{k}))) - A_i(\pi(c(B_i)))|&=|\pi^{\perp}(x)-\pi^{\perp}(c(B_i))|\\&\lesssim \alpha |\pi(c(B_{k}))-\pi(c(B_i))|.
\end{split}\end{equation}
However, Part (2) of Lemma \ref{Bi} ensures that for every $\ell$ with $\pi(c(B_i))\in 2I_{\ell}$, 
$$|A_{\ell}(\pi(c(B_{i})))-A_{i}(\pi(c(B_{i})))|\lesssim \bdary r(B_i).
$$
By the same logic this inequality also holds with $i$ replaced by $k$.  Therefore,
$$|A_{k}(\pi(c(B_{k}))) - A_i(\pi(c(B_i)))|\lesssim \alpha |\pi(c(B_{k}))-\pi(c(B_i))|
$$
But, $c(B_i)\in 1000I_i$ (property (2) of Lemma \ref{SizeBall}), so we may use the calculation (\ref{fixone}) to infer that
$$|A(p)-A(\pi(c(B_i)))|= |A(p)-\pi^{\perp}(c(B_i))| \lesssim \alpha \diam(I_i),
$$
and, similarly,
$$ |A(q)-A(\pi(c(B_{k})))|\lesssim |A(q)-\pi^{\perp}(c(B_{k}))| \lesssim \alpha \diam(I_{k}).
$$
So by the triangle inequality we get
$$|A(p)-A(q)|\lesssim \alpha \bigl[|\pi(c(B_{\ell}))-\pi(c(B_k))|+\diam(I_{\ell})+\diam(I_{k})\bigl]\lesssim \alpha|p-q|.
$$

If instead it holds that $q\in \pi(\Z)$ then recall that $A(q)=\pip(x)$ for $q=\pi(x)$, and in the previous calculation we may replace the role of $(c(B_{k}), r(B_{k}))$ with the pair $(x,t)$ where $q=\pi(x)$ and $t<\diam(I_i)$.  Then Lemma \ref{Slope} yields that $|A_i(\pi(c(B_i)))-A(q)|\lesssim \alpha |c(B_i)-q|$ and the desired estimate follows from repeating estimates from Case 2 above.
  \end{proof}
  
  \begin{lemma}\label{secondderivative}  If $p \in 2I_i$ then 
  $$| A''(p)| \lesssim \frac{\bdary}{\diam I_i}\lesssim \frac{\bdary}{D(p)}.$$
  \end{lemma}
  
  \begin{proof}  We mimic the calculation in Lemma 3.13 of \cite{L}.   Observe that
  $$A''(p) = \sum_{j}A_j'(p)\psi_j'(p)+ \sum_j A_j(p) \psi_j''(p).
  $$ 
 Since $\sum_{j} \psi'_j = \sum_j \psi''_j=0$, we have  
  $$|A''(p)|\leq \sum_j |A_j'(p)-A_i'(p)||\psi_j'(p)|+\sum_j |A_j(p)-A_i(p)||\psi_j''(p)|
  $$
  For each $j$ with $2I_j\cap 2I_i\neq \varnothing$, part (3) of Lemma \ref{Bi} ensures 
  $$|(A_j-A_i)'|\lesssim\bdary,\text{ and }|(A_j(p)-A_i(p))|\lesssim \bdary\diam(I_j).
  $$
    The result follows using the fact that the intervals $2I_j$ have bounded overlap, and the properties of the partition of unity $\psi_j$.
  \end{proof}
  
  \subsection{Localization of $A$}
  We set $\psi\equiv 1$ on $\tfrac{3}{2}I_0$ with $\supp(\psi)\subset 2I_0$.
  
  We define the function $\A:D_0\to D_0^{\perp}$, $$
  \A=\begin{cases} \; \psi\cdot A \text{ on }3I_0 ,\\
 \;0 \text{ on }D_0 \setminus 3I_0.\end{cases}\,$$  
  
  \begin{lemma}\label{Alocal}  The function $\A$ is $C\alpha$-Lipschitz, and
  $$|\A''(p)|\lesssim \frac{\bdary}{D(p)}.
  $$
  \end{lemma}
  
  This result verifies property (1) of Proposition \ref{LipGraphProp}.
  
  \begin{proof} The function $A$ is $C\alpha$-Lipschitz on $3I_0$, and $\sup_{p\in 3I_0}|A(p)|\lesssim \bdary$ (see Lemma \ref{unitscaleAclosetoD0}).  Since $\|\psi\|_{\Lip}\lesssim 1$ and $\supp(\psi)\subset 2I_0$, we infer that $\A$ is $C\alpha$-Lipschitz ($\bdary\ll \alpha$).  Regarding the second derivative property, if $\psi'(p)\neq 0$ or $\psi''(p)\neq 0$, then $\dist(p, I_0)\gtrsim  1$, so $\diam(I_i)\gtrsim 1$ for any $I_i$ with $p\in 2I_i$.  But then if $B_i\in S$ is the ball associated to $I_i$, $r(B_i)\gtrsim 1$ and so Corollary \ref{stopclosetoB0} ensures that both
  $$|A_i'(p)|\lesssim \bdary\text{ and }|A_i(p)|\lesssim \bdary.
  $$
  There are at most a constant number of intervals $I_i$ such that $p\in 2I_i$ so we get that
  $$|A(p)|\lesssim\bdary\text{ and }|A'(p)|\lesssim \bdary.
  $$
  (The first property of course also follows from Lemma \ref{unitscaleAclosetoD0}.)  Since $\|\psi'\|_{\infty}+\|\psi''\|_{\infty}\lesssim 1$, we obtain the desired bound from Lemma \ref{secondderivative}. \end{proof}

  \subsection{Concentration around the graph of $\A$}
  In this section we prove that every point in $F$ will be very close to the graph $\Gamma$ of $\wt{\A}$, defined as $\widetilde{\A}(p)=(p,\A(p))$ for $p \in \R$.  Let us first record an immediate consequence of Lemma \ref{unitscaleAclosetoD0}.
  
  \begin{cor}\label{unitscalebeta} One has  $$\Gamma \subset \Bigl\{\dist(\cdot, D_0)\lesssim \bdary\Bigl\}.$$

  \end{cor}
 
 Indeed, Lemma \ref{unitscaleAclosetoD0} ensures that $\Gamma \cap \pi^{-1}(3I_0)\subset \{\dist(\cdot, D_0)\lesssim\bdary\}$, but outside of $3I_0$ we have that $\A(p)=0$.    This result verifies property (2) of Proposition \ref{LipGraphProp}.\\

  We now move onto verifying property (3) of Proposition \ref{LipGraphProp}.  In view of Lemma \ref{Dcompar}, this property is an immediate consequence of the following lemma.
  
  \begin{lemma} \label{Ftilde}
For every $x \in F$ the following is satisfied:
 $$|x - \wt\A(\pi(x))|\lesssim \bdary d(x).$$
 \end{lemma}

 \begin{proof}
 Certainly $\Z \subset \Gamma$, and so we may assume that $x\in F\backslash \Z$. Lemma \ref{G} then ensures that $\pi(x) \not\in \pi(\mathcal{Z})$.  First suppose $p\in \tfrac{3}{2}I_0$, so that
 \begin{align*}
     |x-\wt\A(\pi(x))|& =|\pi^\perp(x)-A(\pi(x))|\\
     & = \Bigl|\pi^\perp(x)-\sum_i \psi_i(\pi(x)) A_i(\pi(x))\Bigl| \\
     & \leq \sum_i \psi_i(\pi(x))\Bigl|\pi^\perp(x)-A_i(\pi(x))\Bigl|.
 \end{align*}
 
 If $\psi_i(\pi(x))\neq 0$ then $\pi(x) \in 3I_i$ and Lemma \ref{G} ensures that $x \in CB_i$, while from the definition of $I_i$  and Lemma \ref{Dcompar} we find that $\ell(I_i)\approx r(B_i)\approx d(x)$. 
 
 Since $CB_i\in S$ we can find a line $D$ in $\G_{c(B_i)}$ such that $\angle(D, D_0)\leq \alpha$ and $\alpha_{\mu, D}(CB_i)\leq \eps$.   Lemma \ref{notmuchturn} (applied with $L$ an absolute constant) yields that
 $$\angle(D, D_i)\lesssim \bdary.
 $$ 
 On the other hand, since $x\in F$, Lemma \ref{CorStrip} ensures that $\dist(x, D)\lesssim \bdary r(B_i)$.  Thus, combining these observations yields
 $$\dist(x, D_i)\lesssim \bdary r(B_i)\lesssim \bdary d(x).
 $$
Since $\angle(D_i, D_0)\leq \alpha$, this in turn implies that
 $$\Bigl|\pi^\perp(x)-A_i(\pi(x))\Bigl|\lesssim \bdary d(x).
 $$
 On the other hand, if $p\notin\frac{3}{2}I_0$, we have that $d(x)\gtrsim 1$, so the desired estimate is an immediate consequence of Corollaries \ref{CorStripUnit} and \ref{unitscalebeta}.
 The proposition is proved.
 \end{proof}
 
 It remains to verify the final property in Proposition \ref{LipGraphProp}, which we restate here:
 
 \begin{lemma}
 If $B(x,r)\in S$ and $D\in \G_x$ satisfies $\angle (D,D_0)\leq \alpha$ and $\alpha_{\mu,D}(B(x,r))\leq \eps$, then for every $p\in \pi(B(x,r))$, 
$$\dist(\wt\A(p), D)\lesssim \bdary \cdot r.
$$
 \end{lemma}

 \begin{proof}  We first consider the case when $p\in \pi(\Z)$.  Then $\A(p)\in F$ and so Lemma \ref{CorStrip} ensures that $\dist(\A(p), D)\lesssim \bdary \cdot r$.  If $p\notin\pi(\Z)$, then $p\in 2I_i$ for some $i$.  First suppose that $p\in \frac{3}{2}I_0$.  Notice that $r\geq d(x)\gtrsim D(\pi(x))$ (where Corollary \ref{Dcompar} has been used in the second inequality). Since the function $D(p)$ is 1-Lipschitz, it follows that $D(p)\lesssim r$ and so by construction $\ell(I_i)\lesssim r$.    Therefore, from Lemma \ref{SizeBall}, $\pi(B_i)\subset \pi (B(x,Cr))$ and therefore from Lemma \ref{CorStrip} $B_i\cap B(x, 3Cr)\neq \varnothing$.  But now from Lemma \ref{notmuchturn}, $\dist(y, D)\lesssim \bdary \cdot r$ for every $y\in D_i\cap B_i$.  Insofar as $p\in \frac{3}{2}I_0$, $\wt{A}(p)$ is a convex combination of points on the lines $D_i$ where $p\in 2I_i$, the result follows.
 
 Finally, if $p\notin\frac{3}{2}I_0$, then $r\gtrsim 1$ ($x\in  \overline{B_0}$).  In this case the result follows from Corollary \ref{unitscalebeta} and Corollary \ref{stopclosetoB0}.
 \end{proof}

\end{document}